\documentclass{article}

\usepackage[english]{babel}
\usepackage[letterpaper,top=2cm,bottom=2cm,left=3cm,right=3cm,marginparwidth=1.75cm]{geometry}
\usepackage{amsmath, amssymb, amsthm, amsfonts}
\usepackage{graphicx}
\usepackage[colorlinks=true, allcolors=blue]{hyperref}
\usepackage{amssymb}
\usepackage{tikz-cd}
\usepackage{cleveref}
\Crefname{Thm}{Theorem}{Theorems}
\usepackage{ytableau}
\usepackage{calc}
\usepackage{todonotes}
\usepackage{enumitem}
\setlist[enumerate,1]{label=\textnormal{(\arabic*)}, leftmargin=2.5em, itemsep=0.5em}
\setlist[enumerate,2]{label=\textnormal{(\alph*)}, leftmargin=*, itemsep=0.3em}
\usepackage{multirow}
\usepackage{extarrows}
\usepackage{array}

\newcommand\A{{\textstyle\bigwedge}}
\newcommand\Sm{\operatorname{Sm}}
\newcommand\gon{\operatorname{gon}}

\newcommand\coker{\operatorname{coker}}
\newcommand\im{\operatorname{im}}
\newcommand\GL{\operatorname{GL}}
\newcommand\Gr{\operatorname{Gr}}

\newcommand\id{\operatorname{id}}

\newcommand\rank{\operatorname{rank}}
\newcommand\sgn{\operatorname{sgn}}
\newcommand\II{\mathcal I}
\newcommand\OO{\mathcal O}

\newcommand\edet{\operatorname{edet}}

\newcommand\Pic{\operatorname{Pic}}
\newcommand\pr{\operatorname{pr}}

\tikzset{
  symbol/.style={
    draw=none,
    every to/.append style={
      edge node={node [sloped, allow upside down, auto=false]{$#1$}}}
  }
}

\numberwithin{equation}{section}

\theoremstyle{plain}
\newtheorem{Thm}{Theorem}[section]
\newtheorem{Prop}[Thm]{Proposition}
\newtheorem{Cor}[Thm]{Corollary}
\newtheorem{Lem}[Thm]{Lemma}
\newtheorem{Conj}[Thm]{Conjecture}

\theoremstyle{definition}
\newtheorem{Def}[Thm]{Definition}
\newtheorem{Q}[Thm]{Question}
\newtheorem{Notation}[Thm]{Notation}

\theoremstyle{remark}
\newtheorem{Rmk}[Thm]{Remark}
\newtheorem{Ex}[Thm]{Example}

\title{Quadratic equations of tangent varieties via four-way tensors of linear forms}
\author{Junho Choe}

\begin{document}
\maketitle

\begin{abstract}
In the present paper we construct quadratic equations and linear syzygies for tangent varieties using $4$-way tensors of linear forms and generalize this method to higher secant varieties of higher osculating varieties. 
Such equations extend the classical determinantal ones of higher secant varieties and span all the equations of the same degree for smooth projective curves completely embedded by sufficiently positive line bundles, proving a variant of the Eisenbud–Koh–Stillman conjecture on determinantal equations. 
On the other hand, our syzygies are compatible with the Green-Lazarsfeld classes and generate the corresponding Koszul cohomology groups for Segre varieties with a prescribed number of factors. 
To obtain these results we describe the equations of minimal possible degrees and reinterpret the Green-Lazarsfeld classes from the perspective of representation theory.
\end{abstract}

\tableofcontents

\section{Introduction}

Throughout this article we work over the field $\mathbb C$ of complex numbers, a \emph{variety} and a \emph{curve} are irreducible and reduced, and $\mathbb PV$ denotes the projective space of $1$-dimensional subspaces of $V^\ast$, hence $V$ is the space of linear forms on $\mathbb PV$. Let $X\subseteq\mathbb PV$ be a projective variety that is nondegenerate, which means that its linear span is $\langle X\rangle=\mathbb PV$. Write $I(X)=\bigoplus_{j\geq 2}I(X)_j$ for the ideal of $X\subseteq\mathbb PV$.

Our focus is on the \emph{tangent variety}
$$
\tau X=\overline{\bigcup_{z\in\Sm X}\mathbb T_zX}\subseteq\mathbb PV
$$
to $X\subseteq\mathbb PV$ and its quadratic equations, where $\Sm X$ is the smooth locus of $X$, and $\mathbb T_zX$ is the projective tangent space to $X$ at $z\in\Sm X$. More generally, we consider the \emph{$q$-secant variety} $\sigma_q\tau^kX$ to the $\emph{$k$-osculating variety}$ $\tau^kX$ to $X\subseteq\mathbb PV$ for arbitrary integers $q\geq 1$ and $k\geq 0$. 
Our treatment applies not only to $\tau X$ but also to $\sigma_q\tau^kX$, and in fact, taking $\sigma_q\tau^kX$ into account is technically required. See \Cref{Sec:TanSec} for the definition of $\sigma_q\tau^kX$, and note that $\sigma_1\tau^1X=\tau X$.

Suppose that we are given a linear map of the form
$$
T:V_1\otimes\cdots\otimes V_{2k+2}\to V,
$$ 
namely a \emph{$(2k+2)$-way tensor of linear forms}, and associate to it the following composition map $b_{q+1}(T)$.
$$
\begin{tikzcd}
    \A^{q+1}V_1\otimes\cdots\otimes\A^{q+1}V_{2k+2} \ar[rr,"b_{q+1}(T)"] \ar[dr,"\textup{canonical}",hook,swap] & & S^{q+1}V \\
    & S^{q+1}(V_1\otimes\cdots\otimes V_{2k+2}) \ar[ur,"T",swap]
\end{tikzcd}
$$
For example, if $k=0$, then $T$ corresponds to a matrix $M$ with entries in $V$, namely a \emph{matrix of linear forms}, and the image of $b_{q+1}(M)$ is spanned by the $(q+1)$-minors of $M$.

Now our first main theorem presents a systematic construction of quadratic equations for tangent varieties via $4$-way tensors of linear forms. It naturally generalizes the well-known technique \cite[Proposition 1.3]{MR944326} that yields determinantal equations for projective varieties and higher secant varieties using minors of matrices of linear forms.

\begin{Thm}\label{Thm:Main1}
Let $X\subseteq\mathbb PV$ be a nondegenerate projective variety and $T:V_1\otimes\cdots\otimes V_{2k+2}\to V$ be an \emph{$X$-multiplicative} $(2k+2)$-way tensor of linear forms for an integer $k\geq 0$. 
Then for every integer $q\geq 1$ the image of the map $b_{q+1}(T)$ lies in $I(\sigma_q\tau^kX)_{q+1}$, that is, we have a map
$$
b_{q+1}(T):\A^{q+1}V_1\otimes\cdots\otimes\A^{q+1}V_{2k+2}\to I(\sigma_q\tau^kX)_{q+1}.
$$
It is nontrivial whenever $\dim V_i\geq q+1$ for all $1\leq i\leq 2k+2$.
\end{Thm}

For the definition and structure of \emph{$X$-multiplicative} tensors of linear forms, consult with \Cref{Def:XMulti} and \Cref{Prop:StructureXmulti}. 
As a central example, if a very ample line bundle $L$ on $X$ is factored as $L=L_1\otimes\cdots\otimes L_{2k+2}$ in the Picard group $\Pic(X)$, then for the complete embedding $X\subseteq\mathbb PH^0(L)$ the multiplication map
$$
H^0(L_1)\otimes\cdots\otimes H^0(L_{2k+2})\to H^0(L)
$$
is an $X$-multiplicative $(2k+2)$-way tensor of linear forms. For a tensor of this type \Cref{Thm:Main1} works as follows.

\begin{Ex}
Let $T:T^4H^0(\OO_{\mathbb P^1}(1))\to H^0(\OO_{\mathbb P^1}(4))$ be the multiplication map of the fourth tensor power of $H^0(\OO_{\mathbb P^1}(1))$. Using an affine coordinate $x$ on $\mathbb P^1$ we write $H^0(\OO_{\mathbb P^1}(1))=\langle1,x\rangle$ and $H^0(\OO_{\mathbb P^1}(4))=\langle1,\ldots,x^4\rangle=:\langle x_0,\ldots,x_4\rangle$. Then the map $b_2(T)$ is determined by
\begin{align*}
b_2(T)\left((1\wedge x)\otimes\cdots\otimes(1\wedge x)\right) & =\sum_{i_1,\ldots,i_4=0}^1(-1)^{i_1+\cdots+i_4}T\left((x^{i_1}\otimes x^{1-i_1})\otimes\cdots\otimes(x^{i_4}\otimes x^{1-i_4})\right) \\
& =\sum_{i_1,\ldots,i_4=0}^1(-1)^{i_1+\cdots+i_4}x^{i_1+i_2+i_3+i_4}\cdot x^{4-(i_1+\cdots+i_4)} \\
& =\sum_{i=0}^4(-1)^i\binom{4}{i}x^i\cdot x^{4-i} \tag{$i:=i_1+i_2+i_3+i_4$} \\
& =x_0x_4-4x_1x_3+3x_2^2,
\end{align*}
up to scaling, in $S^2H^0(\OO_{\mathbb P^1}(4))$. This is the unique quadratic equation of $\tau\nu_4(\mathbb P^1)\subset\mathbb PH^0(\OO_{\mathbb P^1}(4))=\mathbb P^4$.
\end{Ex}

We point out that \Cref{Thm:Main1} can provide enough quadratic equations to span $I(\tau C)_2$ for an arbitrary completely embedded smooth projective curve $C\subset\mathbb PH^0(L)$ when $L$ has sufficiently large degree. Recall the fact \cite[Theorem 1]{MR944326} that if $C\subset\mathbb PH^0(L)$ has genus $g$ and degree $\geq 4g+2$, then $I(C)$ is generated by the $2$-minors of some matrix of linear forms, together with Eisenbud-Koh-Stillman's problem \cite[Remark on p.\ 518]{MR944326} below. 
\begin{Conj}[Eisenbud-Koh-Stillman conjecture about determinantal equations]
Let $C\subseteq\mathbb PH^0(C,L)$ be a complete embedding of a smooth projective curve of genus $g$, and consider a multiplication map $M:H^0(C,L_1)\otimes H^0(C,L_2)\to H^0(C,L)$ for a suitable decomposition $L_1\otimes L_2=L$ in $\Pic(C)$. If an integer $q\geq 1$ is small compared to a combination of $\deg L_1$, $\deg L_2$, and $g$, then the $(q+1)$-minors of $M$ cut out $\sigma_qC\subseteq\mathbb PH^0(C,L)$ ideal-theoretically.   
\end{Conj}
It has been solved set-theoretically \cite{MR1272376} and scheme-theoretically \cite{MR2712616}. In fact, its ideal-theoretic validity is already well known to experts; thus, the main challenge lies in finding an effective bound.

The following can be considered as an affirmative answer to an analogue of the conjecture for tangent varieties $\tau C\subset\mathbb PH^0(L)$.

\begin{Thm}[Eisenbud-Koh-Stillman type result for tangent varieties]\label{Thm:EKS}
For all integers $q\geq 1$ and $k\geq0$ if $C$ is a smooth projective curve of genus $g$, and if $L$ is its very ample line bundle of degree 
$$
\deg L\geq(2k+2)(2g+q),
$$
then the complete embedding $C\subset\mathbb PH^0(C,L)$ admits a multiplication tensor 
$$
T:H^0(C,L_1)\otimes\cdots\otimes H^0(C,L_{2k+2})\to H^0(C,L),
$$
where $L_1\otimes\cdots\otimes L_{2k+2}=L$ in $\Pic(C)$, such that the induced map $b_{q+1}(T)$ is onto:
    $$
    b_{q+1}(T):\A^{q+1}H^0(C,L_1)\otimes\cdots\otimes\A^{q+1}H^0(C,L_{2k+2})\twoheadrightarrow I(\sigma_q\tau^kC)_{q+1}.
    $$
\end{Thm}

In the case $k=0$ it resolves the Eisenbud-Koh-Stillman conjecture\footnote{We have been informed by Daniele Agostini and Jinhyung Park that they carried out the same study independently. See \cite[Theorems A and B]{agostini2025determinantal}.} since the minimal generators of $I(\sigma_qC)$ have degree $q+1$ whenever $L$ has degree $\geq 2g+2q$ by \cite[Theorem 1.2(2)]{MR4160876}, and its lower bound condition $\deg L\geq(2k+2)(2g+q)$ is very reasonable along the line of Eisenbud-Koh-Stillman's one $\deg L\geq 2(2g+1)$ for the case $(q,k)=(1,0)$.\footnote{If $k=0$, then the lower bound $\deg L\geq 3g+3q+1$ also works by \cite[Theorem A]{agostini2025determinantal}, yielding an improvement when $q$ is small with respect to $g$.}

Furthermore, the quadratic equations of tangent varieties in \Cref{Thm:Main1} lead us to an unexpected byproduct about the nondefectiveness of $\tau X$ and $\sigma_2X$, based on Fulton-Hansen's result \cite[Corollary 4]{MR541334} (cf.\ \cite[Theorem 2]{MR665860}). Compare it with the fact that the conclusion below holds if the embedding line bundle $L$ is $3$-very ample, meaning that it separates every finite subscheme of length $4$ in $X$.

\begin{Cor}\label{Cor:Byproduct}
Let $X$ be a smooth projective variety and $L$ be a very ample line bundle on $X$. Then for the complete embedding $X\subseteq\mathbb PH^0(X,L)$ we obtain
$$
\dim\tau X=2\dim X\quad\text{and}\quad\dim\sigma_2X=2\dim X+1
$$
as long as $L$ allows a decomposition
$$
L=L_1\otimes L_2\otimes L_3\otimes L_4
$$ 
into four line bundles satisfying $\dim|L_i|\geq 1$ for all $1\leq i\leq4$.
\end{Cor}

It can be rephrased as follows: If $L$ is more positive than a product of four pencils, then $\tau X$ and $\sigma_2X$ have the expected dimensions.

The next main theorem tells us that the quadratic equations of $\tau X\subseteq\mathbb PV$ above carry rich \emph{linear syzygies}. Consider the minimal free resolution
$$
\begin{tikzcd}
    F_0 & \ar[l] F_1 & \ar[l] \cdots & \ar[l] F_i=\displaystyle{\bigoplus_{j\in\mathbb Z}}K_{i,j}(I(\tau X))\otimes_{\mathbb C}S(-i-j) & \ar[l] \cdots
\end{tikzcd}
$$
of $I(\tau X)$ over the symmetric algebra $S:=S^\ast V$, where the vector spaces $K_{i,j}(I(\tau X))$, called the \emph{Koszul cohomology groups} of $I(\tau X)$, encode information about bases of the free modules $F_i$. \emph{Linear syzygies} of $\tau X\subseteq\mathbb PV$ are, roughly speaking, elements of the $K_{p,2}(I(\tau X))$ and explain the interaction between the quadratic equations of $\tau X\subseteq\mathbb PV$:
$$
K_{p,2}(I(\tau X))=\ker(\A^pV\otimes I(\tau X)_2\to\A^{p-1}V\otimes I(\tau X)_3)
$$
for a Koszul type differential.

To describe the aforementioned linear syzygies, recall \emph{Schur modules} $S^\lambda U$, of a vector space $U$, indexed by partitions $\lambda=(\lambda_1\geq\lambda_2\geq\cdots)\vdash t$ of integers $t\geq 0$, where if $\lambda=(t)$ (resp.\ $\lambda=(1,\ldots,1)\vdash t$), then we obtain the $t$-th symmetric power $S^tU$ (resp.\ the $t$-th exterior power $\A^tU$). Unless $\lambda_{\dim U}=0$, the Schur module $S^\lambda U$ forms an irreducible finite-dimensional representation of the general linear group $\GL(U)$. Let $p\geq 0$ be an integer and $p_\ast=(p_1,\ldots,p_{2k+2})\vdash p$ be a $(2k+2)$-tuple of nonnegative integers that sum up to $p$:
$$
p_1+\cdots+p_{2k+2}=p,
$$
namely an ordered partition of $p$. Now we assign to $p_\ast$ the tensor product
$$
B_{p_\ast,q+1}(V_1,\ldots,V_{2k+2})=S^{\lambda^1}V_1\otimes\cdots\otimes S^{\lambda^{2k+2}}V_{2k+2},\quad\lambda^i:=(p_i+1,1,\ldots,1)\vdash p+q+1,
$$
of Schur modules of vector spaces $V_1,\ldots,V_{2k+2}$.

\begin{Thm}\label{Thm:Main2}
Let $X\subseteq\mathbb PV$ be a nondegenerate projective variety, and take integers $p\geq 0$, $q\geq 1$, and $k\geq 0$. Then given an $X$-multiplicative $(2k+2)$-way tensor $T:V_1\otimes\cdots\otimes V_{2k+2}\to V$ of linear forms, if 
\begin{equation}\label{Cond:Nonvanishing}
\dim V_1\geq p-p_1+q+1,\quad\ldots,\quad\dim V_{2k+2}\geq p-p_{2k+2}+q+1
\end{equation}
for an ordered partition $p_\ast=(p_1,\ldots,p_{2k+2})\vdash p$, then there is a nonzero natural map
$$
b_{p_\ast,q+1}(T):B_{p_\ast,q+1}(V_1,\ldots,V_{2k+2})\to K_{p,q+1}(I(\sigma_q\tau^kX)).
$$
\end{Thm}

Note that for every integer $j<q+1$ the component $I(\sigma_q\tau^kX)_j$ always vanishes, hence $K_{i,j}(I(\sigma_q\tau^kX))=0$ for all $i\geq0$, and observe that Condition \eqref{Cond:Nonvanishing} is nothing but the nonvanishing of $B_{p_\ast,q+1}(V_1,\ldots,V_{2k+2})$.

The naturality of $b_{p_\ast,q+1}(T)$ can be demonstrated by the case of Segre varieties of the form 
$$
\mathbb PV_1\times\cdots\times\mathbb PV_{2k+2}\subset\mathbb P(V_1\otimes\cdots\otimes V_{2k+2}),
$$
including the sufficiency.

\begin{Thm}\label{Thm:SegreDecompo}
Let $p\geq 0$, $q\geq 1$, and $k\geq 0$ be integers. Then for a Segre variety $Y=\mathbb PV_1\times\cdots\times\mathbb PV_{2k+2}\subseteq\mathbb P(V_1\otimes\cdots\otimes V_{2k+2})$ with $2k+2$ factors, we have a Schur module decomposition
$$
K_{p,q+1}(I(\sigma_q\tau^kY))=\bigoplus_{p_\ast\vdash p}B_{p_\ast,q+1}(V_1,\ldots,V_{2k+2})
$$
over $\GL(V_1)\times\cdots\times\GL(V_{2k+2})$, where $p_\ast$ runs over all ordered partitions of $p$ with $2k+2$ entries. In particular, 
\begin{enumerate}
    \item\label{ThmItem:SegreDecompoTrivial} $I(\sigma_q\tau^kY)_{q+1}=\A^{q+1}V_1\otimes\cdots\otimes\A^{q+1}V_{2k+2}$, and
    \item\label{ThmItem:SegreDecompoUpper} $K_{p,q+1}(I(\sigma_q\tau^kY))=0$ for all $p>\frac{1}{2k+1}\sum_{i=1}^{2k+2}(\dim V_i-q-1)$.
\end{enumerate}
\end{Thm}

The foregoing results collectively constitute a new \emph{matryoshka structure}. Here, a \emph{matryoshka structure} means a sequence of parallel facts indexed by integers $q\geq 1$ such that the $q$-th one holds for $q$-secant varieties. It is very interesting that higher secant varieties exhibit various matryoshka structures. We refer to \cite{MR4160876}, \cite{MR4394426}, \cite{choe2022determinantal}, \cite{MR4441153}, and \cite{choe2023syzygies} for theorems and \cite{MR4888702} for a conjecture.

One of the motivations for this study is Aprodu-Farkas-Papadima-Raicu-Weyman's proof of \emph{Green's conjecture for general canonical curves}. For any smooth canonical curve $C\subset\mathbb PH^0(\omega_C)$ of degree $g\geq 3$, Green suggested a conjectural criterion for the vanishing/nonvanishing of $K_{i,j}(I(C))$ in terms of the Clifford index of $C$. 
\begin{Conj}[Green's conjecture for canonical curves, {\cite[Conjecture 5.1]{MR739785}}]
Let $C\subset\mathbb PH^0(C,\omega_C)$ be a smooth canonical curve of genus $g\geq 3$ and Clifford index $c$. Then $K_{i,j}(I(C))\neq 0$ if and only if
$$
i\in
\begin{cases}
\{0,\ldots,g-c-3\}    & \textup{when }j=2 \\
\{c-1,\ldots,g-4\}    & \textup{when }j=3 \\
\{g-3\}    & \textup{when }j=4 \\
\emptyset    & \textup{otherwise}.
\end{cases}
$$
\end{Conj}
Its validity for general curves was proved by Voisin (\cite{MR1941089} and \cite{MR2157134}) and by Aprodu-Farkas-Papadima-Raicu-Weyman (\cite{MR4022070}). The method of \cite{MR4022070} employs a general hyperplane section
$$
\tau(\nu_g(\mathbb P^1))\cap H\subset H=\mathbb P^{g-1}
$$
of the tangent variety to the rational normal curve $\nu_g(\mathbb P^1)\subset\mathbb P^g$ of degree $g$. By relying on a semicontinuity argument applied to the moduli space of pseudostable curves of arithmetic genus $g$, a detailed analysis of the linear syzygies of $\tau(\nu_d(\mathbb P^1))\subset\mathbb P^g$ completes the proof.

However, the equations of tangent varieties have been surprisingly unexplored in the literature except certain special cases. 
Settling a conjecture \cite[Conjecture 7.6]{MR2363430} of Landsberg and Weyman, Oeding and Raicu \cite[Theorems B, 5.4, and 5.6]{MR3240996} completely described the equations of tangent varieties to Segre-Veronese varieties, and in particular, they showed that for every $4$-factor Segre variety $Y:=\mathbb PV_1\times\cdots\times\mathbb PV_4\subseteq\mathbb P(V_1\otimes\cdots\otimes V_4)$, the quadrics passing through $\tau Y$ form the space 
$$
\A^2V_1\otimes\cdots\otimes\A^2V_4\subset S^2(V_1\otimes\cdots\otimes V_4).
$$
For other projective varieties, only case-by-case results (such as \cite{MR2363430}) are known for tangent varieties, in contrast to higher secant varieties which admit various general frameworks for finding equations, including the determinantal approach.

Another motivation arises from \emph{Green-Lazarsfeld classes}. For a complete embedding $X\subseteq\mathbb PH^0(L)$ let $M:H^0(L_1)\otimes H^0(L_2)\to H^0(L)$ be a multiplication matrix of complete linear systems on $X$ with $h^0(L_1)\geq 2$ and $h^0(L_2)\geq 2$. Then the $2$-minors of $M$ vanish on $X\subseteq\mathbb PH^0(L)$, which means that the induced map
$$
\A^2H^0(L_1)\otimes\A^2H^0(L_2)\to S^2H^0(L)
$$
factors through $I(X)_2$. The determinantal quadrics allow linear syzygies of $X\subset\mathbb PH^0(L)$ to go far away, yielding the nonvanishing
$$
K_{r_1+r_2-2,2}(I(X))\neq 0,\quad r_i:=\dim|L_i|\geq 1,
$$
from an explicit nonzero element, namely a \emph{Green-Lazarsfeld class} \cite[Appendix]{MR739785}. (See \cite{MR1012666}, \cite{MR1238043}, and \cite{MR2157134} for alternative accesses.) Such a construction also works well for higher secant varieties:
$$
K_{r_1+r_2-2q,q+1}(I(\sigma_qX))\neq0
$$
whenever $r_1\geq q$ and $r_2\geq q$. The nonvanishing for $I(X)$ is in fact the origin of Green's conjecture and a former conjecture called the \emph{gonality conjecture} which play a central role in the theory of syzygies, and the one for $I(\sigma_qX)$ generalizes the gonality conjecture to higher secant varieties in terms of the gonality sequence \cite[Theorem 1.1]{choe2023syzygies}.

\section{Preliminaries}

In this section we collect useful notions and facts.

\begin{Notation}
The following are basic.
\begin{enumerate}
    \item $\mathbb N=\{0,1,\ldots\}\subset\mathbb Z$.
    \item For a vector $\alpha=(\alpha_0,\ldots,\alpha_{d-1})\in\mathbb N^d$ we put $|\alpha|=\alpha_0+\cdots+\alpha_{d-1}$ and write $\alpha\vdash|\alpha|$.
    \item Also, define $\alpha!=\alpha_0!\cdots\alpha_{d-1}!$ for such a vector $\alpha$.
    \item For a finite subset $\alpha=\{\alpha_0<\cdots<\alpha_{a-1}\}\subset\mathbb N$ we denote by $|\alpha|$ the number $a$ of elements.
    \item $S^\ast U=\bigoplus_{j\in\mathbb Z}S^jU$ is the symmetric algebra, and $\A^\ast U=\bigoplus_{j\in\mathbb Z}\A^jU$ is the exterior algebra.
    \item $\GL(U)$ is the general linear group of $U$.
    \item $T^{q+1}U$ is the $(q+1)$-th tensor power of $U$.
    \item For a graded module $M$ we write $M_j$ for its component in degree $j\in\mathbb Z$.
    \item $\Sm X$ is the smooth locus of $X$, and $\widehat{X}\subseteq V^\ast$ is the affine cone of $X\subseteq\mathbb PV$.
    \item $\mathfrak{S}_{q+1}$ is the symmetric group of degree $q+1$.
\end{enumerate}
\end{Notation}

\subsection{Multilinear algebra}

Let $U$ be a vector space of finite dimension. Fix a basis $u=(u_0,u_1,\ldots)$ of $U$, and take its dual basis 
$$
\partial_u=(\partial_u^0,\partial_u^1,\ldots)\quad\text{or}\quad\partial^u=(\partial^u_0,\partial^u_1,\ldots)
$$ 
for $U^\ast$, where the former (resp.\ the latter) applies when dealing with $S^\ast U$ (resp.\ $\A^\ast U$). For an element $\alpha=(\alpha_0,\alpha_1,\ldots)\vdash a$ in $\mathbb N^{\dim U}$ with $\alpha_0+\alpha_1+\cdots=a$,
\begin{center}
$S^aU$ has a basis $u^\alpha:=u_0^{\alpha_0}u_1^{\alpha_1}\cdots$,\quad and $S^aU^\ast$ has a basis $\partial_u^\alpha:=(\partial_u^0)^{\alpha_0}(\partial_u^1)^{\alpha_1}\cdots$.    
\end{center} 
Similarly, for a subset $\alpha=\{\alpha_0<\alpha_1<\cdots\}$ of cardinality $a$ in $\{0,\ldots,\dim U-1\}$,
\begin{center}
$\A^aU$ has a basis $u_\alpha:=u_{\alpha_0}\wedge u_{\alpha_1}\wedge\cdots$,\quad and $\A^aU^\ast$ has a basis $\partial^u_\alpha:=\partial^u_{\alpha_0}\wedge \partial^u_{\alpha_1}\wedge\cdots$.    
\end{center}

Just as $S^aU^\ast$ forms the space of partial differentials of order $a$ on $S^\ast U$, so does $\A^aU^\ast$ on $\A^\ast U$ in the following way. For any elements $f\in\A^bU$ and $g\in\A^\ast U$ we have
$$
\partial^u_i(f\wedge g)=(\partial^u_if)\wedge g+(-1)^bf\wedge(\partial^u_ig),
$$
and the usual product on $\A^\ast U^\ast$ corresponds to the composition in reverse order, that is,
$$
\partial^u_\alpha\wedge\partial^u_\beta=\partial^u_\beta\circ\partial^u_\alpha
$$ 
for all $\alpha,\beta\subseteq\{0,\ldots,\dim U-1\}$. 

Consider the \emph{coproduct maps}
$$
\Delta:S^{a+b}U\to S^aU\otimes S^bU\quad\text{and}\quad\Delta:\A^{a+b}U\to\A^aU\otimes\A^bU
$$
which are computed by the formula
$$
\binom{a+b}{a}\Delta(f)=
\begin{cases}
\displaystyle{\sum_{\alpha\vdash a}\frac{1}{\alpha!}u^\alpha\otimes\partial_u^\alpha f} & \textup{if }f\in S^{a+b}U \\
\displaystyle{\sum_{|\alpha|=a}u_\alpha\otimes\partial^u_\alpha f} & \textup{if }f\in\A^{a+b}U,
\end{cases}
$$
where either $\alpha\in\mathbb N^{\dim U}$, or $\alpha\subseteq\{0,\ldots,\dim U-1\}$. They are sections of the product maps $S^aU\otimes S^bU\to S^{a+b}U$ and $\A^aU\otimes\A^bU\to\A^{a+b}U$, respectively, which amounts to Euler's homogeneous function theorem.

\begin{Prop}\label{Prop:Euler}
We have
$$
\binom{a+b}{a}f=
\begin{cases}
\displaystyle{\sum_{\alpha\vdash a}\frac{1}{\alpha!}u^\alpha\cdot\partial_u^\alpha f} & \textup{if }f\in S^{a+b}U \\
\displaystyle{\sum_{|\alpha|=a}u_\alpha\wedge\partial^u_\alpha f} & \textup{if }f\in\A^{a+b}U.
\end{cases}
$$
\end{Prop}

\emph{Koszul differentials} are linear transformations
$$
\delta_{a,b}:\A^aU\otimes S^bU\to\A^{a-1}U\otimes S^{b+1}U\quad\text{and}\quad \delta^{a,b}:S^aU\otimes\A^bU\to S^{a-1}U\otimes\A^{b+1}U
$$
defined as the compositions of
$$
\begin{array}{cl}
\begin{tikzcd}
\A^aU\otimes S^bU \ar[r,"\Delta\otimes\id"] & \A^{a-1}U\otimes U\otimes S^bU \ar[r,"\id\otimes\cdot"] & \A^{a-1}U\otimes S^{b+1}U    
\end{tikzcd} 
& \text{and} \\
\begin{tikzcd}
S^aU\otimes\A^bU \ar[r,"\Delta\otimes\id"] & S^{a-1}U\otimes U\otimes\A^bU \ar[r,"\id\otimes\wedge"] & S^{a-1}U\otimes\A^{b+1}U,
\end{tikzcd}
\end{array}
$$
respectively. They constitute chain complexes
$$
\begin{array}{cl}
     \begin{tikzcd}
        \cdots \ar[r] & \A^{a+1}U\otimes S^{b-1}U \ar[r,"\delta_{a+1,b-1}"] & \A^aU\otimes S^bU \ar[r,"\delta_{a,b}"] & \A^{a-1}U\otimes S^{b+1}U \ar[r] & \cdots
    \end{tikzcd} & \text{and} \\
    \begin{tikzcd}
        \cdots \ar[r] & S^{a+1}U\otimes\A^{b-1}U \ar[r,"\delta^{a+1,b-1}"] & S^aU\otimes\A^bU \ar[r,"\delta^{a,b}"] & S^{a-1}U\otimes\A^{b+1}U \ar[r] & \cdots
    \end{tikzcd} 
\end{array}
$$
that are exact everywhere except at $\A^0U\otimes S^0U$ and $S^0U\otimes\A^0U$, respectively, at which the (co)homology groups have dimension one. Write
$$
Z_{a,b}(U)=\ker\delta_{a,b}\quad\text{and}\quad Z^{a,b}(U)=\ker\delta^{a,b}
$$
for the groups of cycles.

\emph{Schur modules} $S^\lambda U$ are finite-dimensional representations of the general linear group $\GL(U)$, depending on partitions $\lambda=(\lambda_1\geq\lambda_2\geq\cdots)$, where $S^\lambda U=S^tU$ if $\lambda=(t)$, and $S^\lambda U=\A^tU$ if $\lambda=(1,\ldots,1)\vdash t$. 
A Schur module $S^\lambda U$ is nonzero if and only if $\lambda_{\dim U}\neq 0$, and in this case it is irreducible over $\GL(U)$, which means that it does not carry a proper nontrivial subrepresentation. 

\begin{Prop}\label{Prop:TrivialitySchurtoSchur}
If $\phi:S^\lambda U\to S^\mu U$ is a nonzero $\GL(U)$-module homomorphism between Schur modules, then $\lambda$ and $\mu$ are the same, and $\phi$ is just a scalar multiplication. Moreover, the same thing holds for $S^{\lambda^1}V_1\otimes\cdots\otimes S^{\lambda^\ell}V_\ell$ over $\GL(V_1)\times\cdots\times\GL(V_\ell)$.
\end{Prop}

A partition $\lambda \vdash t$ is typically identified with a \emph{Young diagram} consisting of $t$ boxes. These boxes are arranged in a matrix format with one box placed at the $(i,j)$-entry for each $i\geq1$ and $1\leq j\leq\lambda_i$. For example, if $\lambda=(2,1,1)$, then
\begin{equation}\label{Eqn:Y211}
\ydiagram{2,1,1}    
\end{equation}
is the corresponding Young diagram. Transposing the Young diagram of a given partition $\lambda$, we obtain the \emph{conjugate partition} $\lambda'$ of $\lambda$. For the example above we have $\lambda'=(3,1)$ with the following Young diagram.
\begin{equation}\label{Eqn:Y31}
\ydiagram{3,1}    
\end{equation}
For our purpose \emph{hook diagrams} are of special interest. They are Young diagrams that have at most one box in all but the first row, that is, $\lambda_2\leq1$, such as \eqref{Eqn:Y211} and \eqref{Eqn:Y31}. 
Hook diagrams are related to Koszul differentials via
$$
Z_{a,b}(U)=S^{(b,1,\ldots,1)}U\quad\text{and}\quad Z^{a,b}(U)=S^{(a+1,1,\ldots,1)}U
$$
for the partitions $(b,1,\ldots,1)$ and $(a+1,1,\ldots,1)$ of $a+b$.

A fundamental topic in representation theory is to describe the induced $\GL(U)\times\GL(W)$-module structure of a given representation of $\GL(U\otimes W)$. Basic formulas along this line are as follows.

\begin{Thm}[Cauchy-Littlewood formula]\label{Thm:CL}
Let $\lambda$ stand for a partition of an integer $t\geq 0$. Then we have decompositions
$$
\A^t(U\otimes W)=\bigoplus_{\lambda\vdash t}S^\lambda U\otimes S^{\lambda'}W\quad\text{and}\quad S^t(U\otimes W)=\bigoplus_{\lambda\vdash t}S^\lambda U\otimes S^\lambda W
$$
over $\GL(U)\times\GL(W)$. Also, the Schur module $S^\lambda(U\otimes W)$ has summands 
$$
\A^tU\otimes S^{\lambda'}W,\quad S^{\lambda'}U\otimes\A^tW,\quad S^tU\otimes S^\lambda W,\quad\text{and}\quad S^\lambda U\otimes S^tW,
$$
and the other summands can not have the factors $\A^tU$, $\A^tW$, $S^tU$, and $S^tW$, respectively.
\end{Thm}

Consider the identification
$$
\begin{array}{rcl}
    T^{q+1}U\otimes T^{q+1}W & \xlongrightarrow{\cong} & T^{q+1}(U\otimes W), \\
    (u_0 \otimes\cdots\otimes u_q)\otimes(w_0\otimes\cdots\otimes w_q) & \longmapsto & (q+1)!(u_0\otimes w_0)\otimes\cdots\otimes(u_q\otimes w_q).
\end{array}
$$
Then its restriction to $\A^{q+1}U\otimes\A^{q+1}W$, under the averaging embeddings $\A^{q+1}U\subset T^{q+1}U$ and $\A^{q+1}W\subset T^{q+1}W$, is nothing but the determinant map
$$
\det:\A^{q+1}U\otimes\A^{q+1}W\to S^{q+1}(U\otimes W)
$$
for which if $(s_0\wedge\cdots\wedge s_q)\otimes(t_0\wedge\cdots\wedge t_q)\in\A^{q+1}U\otimes\A^{q+1}W$ is viewed as a matrix $(x_{i,j})$ with entries $x_{i,j}:=s_i\otimes t_j\in U\otimes W$, then
$$
\det(x_{i,j})=\sum_{\sigma\in\mathfrak{S}_{q+1}}(\sgn\sigma)x_{\sigma(0),0}\cdots x_{\sigma(q),q}.
$$
Compared to it another restriction, denoted by
$$
\edet:S^{q+1}U\otimes\A^{q+1}W\to\A^{q+1}(U\otimes W),
$$
receives much less attention, but it actually appears in the literature under the name ``exterior minor" \cite{MR1688437} or ``permanent." For a given form $(s_0\cdots s_q)\otimes(t_0\wedge\cdots\wedge t_q)\in S^{q+1}U\otimes\A^{q+1}W$, the map $\edet$ sends the matrix $x_{i,j}:=s_i\otimes t_j$ to
$$
\edet(x_{i,j})=\sum_{\sigma\in\mathfrak{S}_{q+1}}x_{\sigma(0),0}\wedge\cdots\wedge x_{\sigma(q),q}.
$$

The maps $\det$ and $\edet$ enjoy the so-called Jacobi's formula. Note that for fixed bases $u_0,u_1,\ldots$ and $w_0,w_1,\ldots$ of $U$ and $W$, respectively, the tensors $u_i\otimes w_j$ form a basis of $U\otimes W$ so that its dual basis $\partial/\partial(u_i\otimes w_j)$ acts as partial differentials on both $S^\ast(U\otimes W)$ and $\A^\ast(U\otimes W)$.

\begin{Prop}[Jacobi's formula]\label{Prop:Jacobi}
Let $g\in\A^{q+1}W$ be a form. If $f\in\A^{q+1}U$, then
$$
\frac{\partial}{\partial(u_i\otimes w_j)}\det(f\otimes g)=\det(\partial^u_if\otimes\partial^w_jg),
$$
and if $f\in S^{q+1}U$, then
$$
\frac{\partial}{\partial(u_i\otimes w_j)}\edet(f\otimes g)=\edet(\partial_u^if\otimes\partial^w_jg).
$$
\end{Prop}

\begin{proof}
    The formula for $\det$ is well known. Let us see the $\edet$ part. By linearity and symmetry we may assume that $f=u_0^{q+1}$ and $g=w_0\wedge\cdots\wedge w_q$. (Notice that $S^{q+1}U$ is spanned by the $(q+1)$-th powers of vectors in $U$.) Now the case $i>0$ or $j>q$ is trivial, and setting $i=0$ and $0\leq j\leq q$, we have
    \begin{align*}
        \frac{\partial}{\partial(u_0\otimes w_j)}\edet(f\otimes g) & =\frac{\partial}{\partial(u_0\otimes w_j)}\edet\left(u_0^{q+1}\otimes(w_0\wedge\cdots\wedge w_q)\right) \\
        &= (q+1)!\frac{\partial}{\partial(u_0\otimes w_j)}(u_0\otimes w_0)\wedge\cdots\wedge(u_0\otimes w_q) \\
        & =(-1)^j(q+1)!(u_0\otimes w_0)\wedge\cdots\wedge\widehat{(u_0\otimes w_j)}\wedge\cdots\wedge(u_0\otimes w_q) \\
        & =(-1)^j(q+1)\edet\left(u_0^q\otimes(w_0\wedge\cdots\wedge\widehat{w_j}\wedge\cdots\wedge w_q)\right) \\
        & =\edet\left(\partial_u^0u_0^{q+1}\otimes\partial^w_j(w_0\wedge\cdots\wedge w_q)\right) \\
        & =\edet(\partial_u^0f\otimes\partial^w_jg).
    \end{align*}
    We are done.
\end{proof}

The following property of $\det$ and $\edet$ is used later in the proof of \Cref{Prop:Product}.

\begin{Lem}\label{Lem:FactorWedge}
Consider the composition of the maps
$$
\begin{tikzcd}
    (U\otimes\A^qU)\otimes\A^{q+1}W \ar[r,"\id\otimes\Delta"] & U\otimes\A^qU\otimes W\otimes\A^qW \ar[d,equal] \\
    & (U\otimes W)\otimes(\A^qU\otimes\A^qW) \ar[r,"\id\otimes\det"] & (U\otimes W)\otimes S^q(U\otimes W) \ar[r,"\cdot"] & S^{q+1}(U\otimes W).
\end{tikzcd}
$$
Then it factors through $\A^{q+1}U\otimes\A^{q+1}W$. Also, its natural counterpart for $\edet$ holds.
\end{Lem}

\begin{proof}
It is easy to see that the composition maps for $\det$ and $\edet$ are both nonzero whenever 
$$
\dim U\geq
\begin{cases}
q+1 & \textup{for }\det \\
1 & \textup{for }\edet
\end{cases}
$$
and $\dim W\geq q+1$. (Keep in mind the Laplace expansion.) The conclusion follows from \Cref{Prop:TrivialitySchurtoSchur} and the Cauchy-Littlewood formula (\Cref{Thm:CL}).    
\end{proof}

We introduce an operation that serves as a fundamental tool for establishing interesting equations, as in \Cref{Thm:Prolong} below, along with its analogue over exterior algebras.

\begin{Def}
Let $U$ be a vector space, $q\geq1$ be an integer, and $B$ be a subspace of $S^{q+1}U$ (resp.\ $\A^{q+1}U$). For an integer $d\geq0$ we define the \emph{$d$-th prolongation} $B^{(d)}$ of $B$ to be the preimage of $S^dU\otimes B$ (resp.\ $\A^dU\otimes B$) under the coproduct map $\Delta:S^{q+d+1}U\to S^dU\otimes S^{q+1}U$ (resp.\ $\Delta:\A^{q+d+1}U\to\A^dU\otimes\A^{q+1}U$):
$$
B^{(d)}=\Delta^{-1}(S^dU\otimes B)\subseteq S^{q+d+1}U\quad\textup{(resp.\ }B^{(d)}=\Delta^{-1}(\A^dU\otimes B)\subseteq\A^{q+d+1}U\textup{)}.
$$
\end{Def}

The following lemma is needed for the proof of a main theorem.

\begin{Lem}\label{Lem:WedgeProlong}
Let $q\geq 1$ and $d\geq 0$ be integers. Then up to the maps $\det$ and $\edet$ if we view $\A^{q+1}U\otimes\A^{q+1}W$ and $S^{q+1}U\otimes\A^{q+1}W$ as subspaces of $S^{q+1}(U\otimes W)$ and $\A^{q+1}(U\otimes W)$, respectively, then we have
$$
(\A^{q+1}U\otimes\A^{q+1}W)^{(d)}=\A^{q+d+1}U\otimes\A^{q+d+1}W\;\text{and}\;(S^{q+1}U\otimes\A^{q+1}W)^{(d)}=S^{q+d+1}U\otimes\A^{q+d+1}W.
$$
Furthermore, we obtain
$$
(\A^{q+1}V_1\otimes\cdots\otimes\A^{q+1}V_\ell)^{(d)}=\A^{q+d+1}V_1\otimes\cdots\otimes\A^{q+d+1}V_\ell
$$
for vector spaces $V_1,\ldots,V_\ell$.
\end{Lem}

\begin{proof}
The first equality just reflects the well-known fact that any $2$-factor Segre variety $Y:=\mathbb PU\times\mathbb PW\subseteq\mathbb P(U\otimes W)$ satisfies
$$
I(\sigma_qY)_{q+1}=\A^{q+1}U\otimes\A^{q+1}W\subset S^{q+1}(U\otimes W)
$$
for all $q\geq 1$ (see \Cref{Sec:TanSec} and \Cref{Thm:Prolong} below). For the second equality we begin with a new vector space $\overline{U}$ of dimension $\geq q+d+1$ which plays an auxiliary role. Then with the help of the first equality, we see that
\begin{align*}
    \A^{q+d+1}\overline{U}\otimes(S^{q+1}U\otimes\A^{q+1}W)^{(d)} & = (\A^{q+1}\overline{U})^{(d)}\otimes(S^{q+1}U\otimes\A^{q+1}W)^{(d)} \\
    & \subseteq (\A^{q+1}\overline{U}\otimes S^{q+1}U\otimes\A^{q+1}W)^{(d)} \tag{\Cref{Prop:Jacobi}} \\
    & \subseteq (\A^{q+1}(\overline{U}\otimes U)\otimes\A^{q+1}W)^{(d)} \\
    & =\A^{q+d+1}(\overline{U}\otimes U)\otimes\A^{q+d+1}W,
\end{align*}
hence applying the Cauchy-Littlewood formula (\Cref{Thm:CL}) to $\A^{q+d+1}(\overline{U}\otimes U)$, one finds that $$
(S^{q+1}U\otimes\A^{q+1}W)^{(d)}\subseteq S^{q+d+1}U\otimes\A^{q+d+1}W,
$$
and the reverse containment is easy to show.

Now we prove the last equality. Let us proceed by induction on $\ell\geq 2$. We have already verified the initial step $\ell=2$. Assume the induction hypothesis for one less than $\ell\geq 3$. Writing
$$
R=
\begin{cases}
\A^\ast(V_1\otimes\cdots\otimes V_{\ell-1}) & \textup{if }\ell\textup{ is even} \\
S^\ast(V_1\otimes\cdots\otimes V_{\ell-1}) & \textup{if }\ell\textup{ is odd},
\end{cases}
$$
we are able to exploit the two previous equalities so that
$$
(\A^{q+1}V_1\otimes\cdots\otimes\A^{q+1}V_{\ell-1}\otimes\A^{q+1}V_\ell)^{(d)}\subseteq(R_{q+1}\otimes\A^{q+1}V_\ell)^{(d)}= R_{q+d+1}\otimes\A^{q+d+1}V_\ell.
$$
Taking each $V_i$ in the place of $V_\ell$ we find that
$$
(\A^{q+1}V_1\otimes\cdots\otimes\A^{q+1}V_\ell)^{(d)}\subseteq\A^{q+d+1}V_1\otimes\cdots\otimes\A^{q+d+1}V_\ell.
$$
As above, the inclusion $\supseteq$ is trivial.
\end{proof}

\subsection{Higher secant varieties to higher osculating varieties}\label{Sec:TanSec}

Let $z\in\Sm X$ be a smooth point of $X\subseteq\mathbb PV$ and $n=\dim X$. Take a smooth lifting 
\begin{equation}\label{Eqn:LocalPara}
\phi=\phi(t)=\phi(t_1,\ldots,t_n):(\mathbb C^n,0)\to (\widehat{X},\widehat{z})\subseteq(V^\ast,\widehat{z})
\end{equation} 
of a local parametrization of $X$ at $z$, that is, passing through the projection map $(V^\ast,\widehat{z})\to(\mathbb PV,z)$, we obtain a local parametrization of the germ $(X,z)$. Then for an integer $k\geq 0$ the \emph{projective $k$-osculating space} to $X\subseteq\mathbb PV$ at $z$ is defined as
$$
\mathbb T^k_zX=\mathbb P\left\langle\frac{\partial^{\alpha_1+\cdots+\alpha_n}}{\partial t_1^{\alpha_1}\cdots\partial t_n^{\alpha_n}}\phi(0)\in V^\ast:\alpha_1+\cdots+\alpha_n\leq k\right\rangle^\ast\subseteq\mathbb PV.
$$ For example, if $k=0$, then $\mathbb T^0_zX$ is the point $\{z\}$, if $k=1$, then $\mathbb T^1_zX$ is the projective tangent space $\mathbb T_zX$ to $X$ at $z$, and if $k=2$, then $\mathbb T^2_zX$ is the projective osculating space to $X$ at $z$.

Then our objects of interest are \emph{higher secant varieties} to \emph{higher osculating varieties}. For integers $q\geq 1$ and $k\geq0$ the \emph{$q$-secant variety} to the \emph{$k$-osculating variety} to $X\subseteq\mathbb PV$ is defined to be the Zariski closure
$$
\sigma_q\tau^kX=\overline{\bigcup_{z_1,\ldots,z_q\in\Sm X}\langle\mathbb T^k_{z_1}X,\ldots,\mathbb T^k_{z_q}X\rangle}\subseteq\mathbb PV,
$$
where $(z_1,\ldots,z_q)\in(\Sm X)^q$ runs over all \emph{general} $q$-tuples of smooth points in $X$ so that $\dim\langle\mathbb T^k_{z_1}X,\ldots,\mathbb T^k_{z_q}X\rangle$ attains the maximum, and $\langle-\rangle$ means the linear span in $\mathbb PV$. If $q=1$, then $\tau^kX:=\sigma_1\tau^kX$ is the \emph{$k$-osculating variety} to $X\subseteq\mathbb PV$, and if in addition $k=1$, then $\tau X:=\tau^1X$ is the \emph{tangent variety} to $X\subseteq\mathbb PV$. If $k=0$, then $\sigma_qX:=\sigma_q\tau^0X$ is the \emph{$q$-secant variety} to $X\subseteq\mathbb PV$. Note that $\sigma_q\tau^kX=\sigma_q(\tau^kX)$.

The problem of finding equations of $\sigma_qX\subseteq\mathbb PV$ has been extensively studied, carrying hugely numerous results. Among them we may consider the following as a starting point. 
\begin{enumerate}
    \item $I(\sigma_qX)$ does not have any nonzero $q$-form.
    \item Its $(q+1)$-forms can be computed via the prolongation.
\end{enumerate}

\begin{Thm}[{\cite[Lemma 2.2]{MR1966752}}, cf.\ {\cite[Theorem 1.2]{MR2541390}}]\label{Thm:Prolong}
For integers $q\geq 1$ and $d\geq 0$ we have
    $$
    I(\sigma_qX)_{q+1}^{(d)}=I(\sigma_{q+d}X)_{q+d+1}
    $$
in $S^{q+d+1}V$. In particular, $I(\sigma_qX)_{q+1}=I(X)_2^{(q-1)}$.
\end{Thm}

As for equations of tangent varieties, Oeding and Raicu have focused on the case of Segre-Veronese varieties
$$
Y:=\mathbb PV_1\times\cdots\times\mathbb PV_\ell\subseteq\mathbb P(S^{d_1}V_1\otimes\cdots\otimes S^{d_\ell}V_\ell).
$$

\begin{Thm}[{\cite[Theorem B]{MR3240996}}]\label{Thm:OR}
Let $Y$ be as above. Then $I(\tau Y)$ is generated as
$$
I(\tau Y)=(I(\tau Y)_2,I(\tau Y)_3,I(\tau Y)_4),
$$
and in this expression the following hold.
\begin{enumerate}
    \item $I(\tau Y)_2$ is unnecessary if and only if $d_1+\cdots+d_\ell\leq 3$.
    \item $I(\tau Y)_3$ is unnecessary if and only if $Y$ belongs to one of the following cases.
    \begin{enumerate}
        \item $\OO_Y(1)=\OO_{\mathbb P^1}(d)$ for some $d\geq 3$ except $d=4$.
        \item $\OO_Y(1)=\OO_{\mathbb P^s\times\mathbb P^1}(1,d)$ for some $s\geq 1$ and $d\geq5$.
        \item $\tau Y=\mathbb P(S^{d_1}V_1\otimes\cdots\otimes S^{d_\ell}V_\ell)$.
    \end{enumerate}
    \item $I(\tau Y)_4$ is unnecessary if and only if $\{d_1,\ldots,d_\ell\}\not\supseteq\{1,1,1\},\{1,2\},\{3\}$ as a multiset.
\end{enumerate}
\end{Thm}

Furthermore, \cite[Section 5]{MR3240996} explicitly describes the minimal generators of $I(\tau Y)$ for any Segre-Veronese variety $Y$. Accordingly, $I(\tau Y)_2$ has the form
\begin{equation}
I(\tau Y)_2=\bigoplus_{\lambda^1,\ldots,\lambda^\ell}S^{\lambda^1}V_1\otimes\cdots\otimes S^{\lambda^\ell}V_\ell,   
\end{equation}
where $\lambda^1\vdash 2d_1,\ldots,\lambda^\ell\vdash 2d_\ell$ run over all partitions of the forms
$$
\lambda^i=(2d_i-a_i\geq a_i)\quad\text{with}\quad a_1+\cdots+a_\ell\in\{4,6,8,\ldots\}.
$$
For example, if $Y=\mathbb PV_1\times\cdots\times\mathbb PV_4\subseteq\mathbb P(V_1\otimes\cdots\otimes V_4)$, then \begin{equation}\label{Eqn:4FactorQuadric}
 I(\tau Y)_2=\A^2V_1\otimes\cdots\otimes\A^2V_4.
\end{equation}
The work of Oeding and Raicu settles a conjecture of Landsberg and Weyman on generators of $I(\tau Y)$ for Segre varieties $Y$, that is, $d_1=\cdots=d_\ell=1$.

We end this subsection with an observation about the equations of the smallest possible degree for $\sigma_q\tau^kX\subseteq\mathbb PV$. Let $\emptyset\neq U\subseteq\Sm X$ be any open dense subset that maximizes $\dim\mathbb T^k_zX$ for all $z\in U$, and put $L=\OO_X(1)|_U$. Then we may view $S^{q+1}V$ as a subspace of $T^{q+1}H^0(L)=H^0(U^{q+1},L^{\boxtimes q+1})$, where $L^{\boxtimes q+1}$ is the box product $L\boxtimes\cdots\boxtimes L$ on the $(q+1)$-th Cartesian power $U^{q+1}$. 
Consider the big diagonal $\Delta^{q+1}:=\bigcup_{0\leq i<j\leq q}\Delta_{i,j}\subset U^{q+1}$, where $\Delta_{i,j}:=\{(z_0,\ldots,z_q)\in U^{q+1}:z_i=z_j\}$. Recall that the $\ell$-th symbolic power $\II_{\Delta^{q+1}/U^{q+1}}^{(\ell)}$ is the ideal sheaf of all functions having vanishing order $\geq\ell$ along $\Delta^{q+1}$ and that $\II_{\Delta^2/U^2}^{(\ell)}$ coincides with the $\ell$-th ordinary power $\II_{\Delta^2/U^2}^l$ since $U$ is smooth.

\begin{Thm}\label{Thm:BottomEqn}
Let $q\geq1$ and $k\geq0$ be integers Then
$$
I(\sigma_q\tau^kX)_{q+1}=S^{q+1}V\cap H^0\left(U^{q+1},L^{\boxtimes q+1}\otimes\II_{\Delta^{q+1}/U^{q+1}}^{(2k+2)}\right)
$$
in $H^0(U^{q+1},L^{\boxtimes q+1})$.
\end{Thm}

\begin{proof}
First, we check the case $q=1$. Let $Q\in S^2V$ stand for a quadratic form. Due to the short exact sequences
$$
\begin{tikzcd}
    0 \ar[r] & \II_{\Delta^2/U^2}^{(i+1)} \ar[r] & \II_{\Delta^2/U^2}^{(i)} \ar[r] & S^i\Omega_{\Delta^2} \ar[r] & 0,
\end{tikzcd}
$$
the following are equivalent.
\begin{enumerate}
\item\label{Item:BottomEqn1} $Q\in H^0\left(L^{\boxtimes 2}\otimes\II_{\Delta^2/U^2}^{(2k+2)}\right)$.
\item\label{Item:BottomEqn2} The symmetric bilinear form of $Q$ satisfies
\begin{equation}\label{Eqn:QuadraticHigherVanishing}
Q(\partial^\beta\phi(t),\partial^\gamma\phi(t)):=Q\left(\frac{\partial^{\beta_1+\cdots+\beta_n}}{\partial t_1^{\beta_1}\cdots\partial t_n^{\beta_n}}\phi(t),\frac{\partial^{\gamma_1+\cdots+\gamma_n}}{\partial t_1^{\gamma_1}\cdots\partial t_n^{\gamma_n}}\phi(t)\right)\equiv0
\end{equation}
for all $\beta=(\beta_1,\ldots,\beta_n)$ and $\gamma=(\gamma_1,\ldots,\gamma_n)$ in $\mathbb N^n$ with $|\beta|\leq k$ and $|\gamma|\leq k$, where $\phi(t)$ is as in \eqref{Eqn:LocalPara}.
\item $Q\in I(\tau^kX)_2$.
\end{enumerate}
For the equivalence of \ref{Item:BottomEqn1} and \ref{Item:BottomEqn2} see \Cref{Lem:BottomEqnLem} below. Thus, we have arrived at
$$
I(\tau^kX)_2=S^2V\cap H^0\left(L^{\boxtimes2}\otimes\II_{\Delta^2/U^2}^{(2k+2)}\right).
$$

Second, we handle the remaining cases $q\geq2$. Note that the coproduct map $\Delta:S^{q+1}V\to S^{q-1}V\otimes S^2V$ can be realized by selecting any pair of factors in $T^{q+1}V$. This implies that
\begin{align*}
    I(\sigma_q\tau^kX)_{q+1} & =I(\tau^kX)_2^{(q-1)} \tag{\Cref{Thm:Prolong}} \\
    & =\left(S^2V\cap H^0\left(L^{\boxtimes2}\otimes\II_{\Delta^2/U^2}^{(2k+2)}\right)\right)^{(q-1)} \tag{by the case $q=1$}\\
    & =S^{q+1}V\cap H^0\left(L^{\boxtimes q+1}\otimes\bigcap_{0\leq i<j\leq q}\II_{\Delta_{i,j}/U^{q+1}}^{(2k+2)}\right) \\
    & =S^{q+1}V\cap H^0\left(L^{\boxtimes q+1}\otimes\II_{\Delta^{q+1}/U^{q+1}}^{(2k+2)}\right).
\end{align*}
\end{proof}

\begin{Lem}\label{Lem:BottomEqnLem}
Let $Q$ be as above. Then the containment
\begin{equation}\label{Eqn:QuadricContainment}
Q\in H^0\left(U^2,L^{\boxtimes 2}\otimes\II_{\Delta^2/U^2}^{(2k+2)}\right)
\end{equation}
is equivalent to the vanishing \eqref{Eqn:QuadraticHigherVanishing} in each of the following ranges.
\begin{enumerate}[label=\textup{(R\arabic*)}, leftmargin=3.5em, itemsep=0.5em]
    \item\label{Item:VanishingOrderFirstStep} $\beta=0$ and $|\gamma|\leq 2k+1$.
    \item\label{Item:VanishingOrderFirstStep2} $|\beta|+|\gamma|\leq 2k+1$.
    \item\label{Item:VanishingOrderFirstStep3} $|\beta|\leq k$ and $|\gamma|\leq k$.
\end{enumerate}

\end{Lem}

\begin{proof}
First, we show that the containment \eqref{Eqn:QuadricContainment} holds if and only if the case \ref{Item:VanishingOrderFirstStep} supports the vanishing \eqref{Eqn:QuadraticHigherVanishing}. The problem is analytic-local: The smooth variety $U^2$ around $\Delta^2$ can be locally parametrized by 
$$
(\phi(t),\phi(t+s))\in(\widehat{X}^2,(\widehat{z},\widehat{z})),
$$
where $s=(s_1,\ldots,s_n)\in(\mathbb C^n,0)$ is a copy of $t=(t_1,\ldots,t_n)$, and the ideal sheaf $\II_{\Delta^2/U^2}$ is generated as 
$$
\II_{\Delta^2/U^2}=(s_1,\ldots,s_n)\subset\OO_{U^2}
$$
up to the analytic localization. So the equivalence in question follows from taking
$$
\left.\frac{\partial^{\gamma_1+\cdots+\gamma_n}}{\partial s_1^{\gamma_1}\cdots\partial s_n^{\gamma_n}}Q(\phi(t),\phi(t+s))\right|_{s=0}
$$
into account.

Second, since the range \ref{Item:VanishingOrderFirstStep2} is bigger than both \ref{Item:VanishingOrderFirstStep} and \ref{Item:VanishingOrderFirstStep3}, the corresponding directions are immediate consequences. 

Finally, for their converses we proceed by finite induction on $0\leq\ell:=|\beta|+|\gamma|\leq 2k+1$. The first step $\ell=0$ is trivial, hence we assume that $\ell\geq1$. Consider the case $\beta_i\geq 1$ for instance. Since $|\beta-e_i|+|\gamma|=\ell-1$, the induction hypothesis gives 
$$
Q(\partial^{\beta-e_i}\phi(t),\partial^\gamma\phi(t))\equiv0.
$$
Differentiating with respect to $t_i$ we obtain 
$$
Q(\partial^\beta\phi(t),\partial^\gamma\phi(t))\equiv-Q(\partial^{\beta-e_i}\phi(t),\partial^{\gamma+e_i}\phi(t))
$$
and its analogues. Therefore, applying the rule successively one can see that
$$Q(\partial^\beta\phi(t),\partial^\gamma\phi(t))\equiv(-1)^{|\beta|-|\beta'|}Q(\partial^{\beta'}\phi(t),\partial^{\gamma'}\phi(t))
$$
for any $\beta',\gamma'\in\mathbb N^n$ with $|\beta'|+|\gamma'|=\ell$. 
The desired result comes from a suitable choice of $(\beta',\gamma')$:
\begin{enumerate}
    \item If we have supposed the condition \ref{Item:VanishingOrderFirstStep}, then we put $(\beta',\gamma')=(0,\beta+\gamma)$.
    \item As for the case \ref{Item:VanishingOrderFirstStep3} if $\ell\leq 2k$, then we take $(\beta',\gamma')$ such that $|\beta'|\leq k$ and $|\gamma'|\leq k$, and if $\ell=2k+1$, then we put $(\beta',\gamma')=(\gamma,\beta)$, where the symmetry of $Q$ is used.
\end{enumerate}
\end{proof}

\begin{Rmk}
At least if $\dim X\leq 2$, then we have
$$
\II_{\Delta^{q+1}/U^{q+1}}^{(\ell)}=\II_{\Delta^{q+1}/U^{q+1}}^\ell
$$
for all $\ell\geq 1$. In dimension one the reason is that the diagonals are just effective divisors, and in dimension two we refer to \cite[Corollary 3.8.3]{MR1839919}.
\end{Rmk}

\Cref{Thm:BottomEqn} reveals a geometric meaning of \eqref{Eqn:4FactorQuadric} and applies extensively to higher secant varieties to the $k$-osculating varieties to $(2k+2)$-factor Segre varieties.

\begin{Cor}\label{Cor:SegreBottomEquation}
Take a Segre variety $Y:=\mathbb PV_1\times\cdots\times\mathbb PV_{2k+2}\subseteq\mathbb P(V_1\otimes\cdots\otimes V_{2k+2})$ with $2k+2$ factors. Then $I(\sigma_q\tau^kY)_{q+1}$ has the form
$$
I(\sigma_q\tau^kY)_{q+1}=\A^{q+1}V_1\otimes\cdots\otimes\A^{q+1}V_{2k+2}
$$
over $\GL(V_1)\times\cdots\times\GL(V_{2k+2})$.
\end{Cor}

\begin{proof}
We exploit \Cref{Thm:BottomEqn} for which we may take $U=Y$ thanks to the homogeneity of $Y$. Then the K{\"u}nneth formula reads $H^0(\OO_Y(1)^{\boxtimes q+1})=H^0(\OO_{\mathbb PV_1}(1)^{\boxtimes q+1})\otimes\cdots\otimes H^0(\OO_{\mathbb PV_{2k+2}}(1)^{\boxtimes q+1})$, and it is rather simple to see that 
$$
H^0\left(\OO_{\mathbb PV_i}(1)^{\boxtimes q+1}\otimes\II_{\Delta^{q+1}_i/(\mathbb PV_i)^{q+1}}^{(\ell)}\right)=
\begin{cases}
    \A^{q+1}V_i & \textup{if }\ell=1 \\
    0 & \textup{if }\ell\geq 2,
\end{cases}
$$
where the $\Delta_i^{q+1}$ are the big diagonals of $(\mathbb PV_i)^{q+1}$, respectively. Now counting the vanishing order along the big diagonal $\Delta^{q+1}\subset Y^{q+1}$, we are done.
\end{proof}

\begin{Rmk}
The proof of \Cref{Cor:SegreBottomEquation} also deduces that no degree $q+1$ hypersurfaces pass through $\sigma_q\tau^kY$ for any Segre variety $Y$ with $\ell<2k+2$ factors.
\end{Rmk}

\subsection{Syzygies}

Let $S=\mathbb C[x_0,\ldots,x_r]$ be a polynomial ring in $r+1$ variables with the standard grading and $M$ be a finitely generated graded $S$-module. Consider its minimal graded free resolution
$$
\begin{tikzcd}
F_0 & \ar[l,"d_1",swap] F_1 & \ar[l,"d_2",swap] \cdots & \ar[l,"d_i",swap] F_i & \ar[l,"d_{i+1}",swap] \cdots    
\end{tikzcd}
$$
so that $\coker\,d_1\cong M$, and $\im d_i\subseteq(x_0,\ldots,x_r)\cdot F_{i-1}$ for all $i\geq 1$, and write
$$
F_i=\bigoplus_{j\in\mathbb Z}K_{i,j}(M)\otimes S(-i-j)
$$
by grouping the minimal homogeneous generators of each degree. Then the vector space $K_{i,j}(M)$ is regarded as the space of \emph{$i$-th syzygies of weight $j$} of $M$. It is isomorphic to the cohomology group of the Koszul type complex
$$
\begin{tikzcd}
    \A^{i+1}V\otimes M_{j-1} \ar[r] & \A^{i}V\otimes M_j \ar[r] & \A^{i-1}V\otimes M_{j+1}
\end{tikzcd}
$$
at the middle, where $V=S_1$ is the space of linear forms in $S$, and we call it the $(i,j)$-th \emph{Koszul cohomology group} of $M$. 

For our study, of special interest are syzygies of the minimal possible weight of $I(\sigma_q\tau^kX)$. Recall that $I(\sigma_q\tau^kX)_q=0$ by the nature of higher secant varieties.

\begin{Def}
Let $q\geq1$ and $k\geq0$ be integers. A \emph{bottom syzygy} of $\sigma_q\tau^kX\subseteq\mathbb PV$ is an element of the $K_{p,q+1}(I(\sigma_q\tau^kX))$, $p\geq 0$.
\end{Def}

As a fundamental example syzygies of higher secant varieties to Segre varieties have been well understood.

\begin{Thm}[Lascoux resolutions, {\cite{MR520233}} and {\cite[Theorem 6.1.4]{MR1988690}}]\label{Thm:Lascoux}
For a Segre variety $Y:=\mathbb PV_1\times\mathbb PV_2\subseteq\mathbb P(V_1\otimes V_2)$ with two factors and an integer $q\geq 1$, we have
$$
K_{p,mq+1}(I(\sigma_qY))=\bigoplus_{\lambda^1,\lambda^2}S^{\mu^1}V_1\otimes S^{\mu^2}V_2
$$
for all $p\geq 0$ and $m\geq1$, where $\lambda^1$ and $\lambda^2$ run over all partitions with $\max\{\lambda^1_1,\lambda^2_1\}\leq m$ and $|\lambda^1|+|\lambda^2|=p-m^2+1$, and $\mu^1$ and $\mu^2$ are ``block" Young diagrams of the forms
$$
\mu^1=\begin{pmatrix}
    \lambda^0 & {\lambda^2}' \\
    \lambda^1 &
\end{pmatrix}
\quad\text{and}\quad\mu^2=
\begin{pmatrix}
    \lambda^0 & {\lambda^1}' \\
    \lambda^2 &
\end{pmatrix}
$$
for the partition $\lambda^0:=(m,\ldots,m)\vdash (q+m)m$. The other Koszul cohomology groups all vanish.
\end{Thm}

Consequently, the bottom syzygies of $\sigma_qY$ for a $2$-factor Segre variety $Y$ form the space
$$
K_{p,q+1}(I(\sigma_qY))=\bigoplus_{a+b=p}S^{(a+1,1,\ldots,1)}V_1\otimes S^{(b+1,1,\ldots,1)}V_2
$$
for partitions $(a+1,1,\ldots,1)$ and $(b+1,1,\ldots,1)$ of $p+q+1$.

We often use Betti tables to encapsulate minimal free resolutions and Koszul cohomology groups. When we restrict ourselves to the case of $I(\tau Y)$ for $4$-factor Segre varieties $Y$, they are diagrams whose entries are $\beta_{i,j}:=\dim K_{i,j}(I(\tau Y))$. In the Betti table the column $i=0$ is controlled by Oeding-Raicu's result, whereas the row $j=2$ is determined by one of our main results:
$$
\begin{array}{c|ccccl}
    j\setminus i    & 0 & 1 & 2 & 3 & \cdots \\ \hline
    2               & \beta_{0,2} & \beta_{1,2} & \beta_{2,2} & \beta_{3,2} & \leftarrow\textup{\Cref{Thm:SegreDecompo}} \\
    3               & \beta_{0,3} & \beta_{1,3} & \beta_{2,3} & \beta_{3,3} & \multirow{2}{*}{$\cdots$} \\
    4               & \beta_{0,4} & \beta_{1,4} & \beta_{2,4} & \beta_{3,4} & \\
    5               & \beta_{0,5} & \beta_{1,5} & \beta_{2,5} & \beta_{3,5} & \\
    \vdots    & \uparrow & \multicolumn{2}{c}{\vdots} \\
                     & \cite{MR3240996} & \phantom{\cite{MR3240996}} & \phantom{\cite{MR3240996}} & \phantom{\cite{MR3240996}}
\end{array} 
$$

\subsection{Tensors of linear forms}

This subsection presents the concept of a \emph{tensor of linear forms}. Recall that an element of a given tensor product $V_1\otimes\cdots\otimes V_\ell$, namely an $\ell$-way tensor, is called simple if it is equal to $v_1\otimes\cdots\otimes v_\ell$ for some nonzero vectors $0\neq v_i\in V_i$, $1\leq i\leq\ell$. 

\begin{Def}\label{Def:XMulti}
Let $\ell\geq 2$ be an integer. An \emph{$\ell$-way tensor of linear forms} is a linear map 
$$
T:V_1\otimes\cdots\otimes V_\ell\to V
$$ 
from a tensor product $V_1\otimes\cdots\otimes V_\ell$ of $\ell$ vector spaces to the space $V$ of linear forms on $\mathbb PV$, and a \emph{matrix of linear forms} simply refers to the case $\ell=2$. For such a tensor $T$ we introduce the following terminology.
\begin{enumerate}
    \item $T$ is \emph{$1$-generic} (see Eisenbud \cite{MR944327}) if $T(v_1\otimes\cdots\otimes v_\ell)\neq 0$ for all $0\neq v_i\in V_i$, $1\leq i\leq\ell$.
    \item $T$ is \emph{$X$-simple} if $T^\ast(\widehat{z})\in V_1\otimes\cdots\otimes V_\ell$ is either simple or zero for a general point $\widehat{z}\in\widehat{X}\subseteq V^\ast$.
    \item $T$ is \emph{$X$-multiplicative} if it is both $1$-generic and $X$-simple.
\end{enumerate}
\end{Def}

Here, $1$-generic matrices of linear forms expand the discussion of generic matrices, that is, the matrices whose entries are pairwise distinct indeterminates.

\begin{Thm}[Part of {\cite[Theorem 2.1 and Corollary 2.2]{MR944327}}]\label{Thm:Eisenbud1Generic}
Let $M:U\otimes W\to V$ be a $1$-generic matrix of linear forms, and suppose that $1\leq e:=\dim W-\dim U+1\leq\dim\mathbb PV$. For the ideal $I_{\dim U}(M)$ generated by the maximal minors of $M$, the following hold.
\begin{enumerate}
    \item It cuts out a subscheme $Z\subset\mathbb PV$ of codimension $e$.
    \item It is prime as long as $\dim Z\geq 1$.
\end{enumerate}
Moreover, for each $1\leq q\leq\dim U-1$ the natural map $\A^{q+1}U\otimes\A^{q+1}W\to S^{q+1}V$ induces an injection $\A^{q+1}U\hookrightarrow S^{q+1}V$ when fixing $0\neq w_0\wedge\cdots\wedge w_q\in\A^{q+1}W$.
\end{Thm}

The $X$-simplicity can be understood in an alternative way.

\begin{Rmk}\label{Rmk:XSimple}
An $\ell$-way tensor $T:V_1\otimes\cdots\otimes V_\ell\to V$ of linear forms is $X$-simple if and only if every ``flattening" $V_i\otimes(V_1\otimes\cdots\otimes V_{i-1}\otimes V_{i+1}\otimes\cdots\otimes V_\ell)\to V$ defines a matrix $M_i$ of linear forms such that $\rank M_i(\widehat{z})\leq 1$ for all $\widehat{z}\in\widehat{X}\subseteq V^\ast$.
\end{Rmk}

\begin{Prop}[Geometric structure of $X$-multiplicative tensors, cf.\ {\cite[Proposition 1.6]{MR944327} and \cite[Lemma 4.3]{choe2022determinantal}}]\label{Prop:StructureXmulti}
Let $T:V_1\otimes\cdots\otimes V_\ell\to V$ be an $X$-multiplicative $\ell$-way tensor of linear forms. Consider a morphism $\phi:U\to X$ of a smooth variety $U$ and the pullback line bundle $L=\phi^\ast\OO_X(1)$ on $U$. Then we may regard the $V_i$ as linear systems $V_i\subseteq H^0(U,L_i)$ for some line bundles $L_i$ on $U$ having the following properties.
\begin{enumerate}
    \item Each $|V_i|$ has no fixed components.
    \item $L_1\otimes\cdots\otimes L_\ell=L(-F)$ for the fixed divisorial part $F\subset U$ of the linear system $|\im T|\subseteq|L|$.
    \item $T$ comes from the multiplication of global sections, that is, the diagram below commutes.
    $$
\begin{tikzcd}
    V_1\otimes\cdots\otimes V_\ell \ar[rr,"T"] \ar[d,phantom, "\rotatebox{-90}{$\subseteq$}" description] & & V \ar[d,phantom, "\rotatebox{-90}{$\subseteq$}" description] \\
    H^0(U,L_1)\otimes\cdots\otimes H^0(U,L_\ell) \ar[r,"\cdot",swap] & H^0(U,L(-F)) \ar[r,"\cdot F",swap] & H^0(U,L)
\end{tikzcd}
$$
\end{enumerate}
\end{Prop}

\begin{proof}
    Fix bases of $V_1,\ldots,V_\ell$ so that in the natural way the tensor $T$ corresponds to a collection of linear forms $T_{j_1,\ldots,j_\ell}\in V\subseteq H^0(L)$, $0\leq j_i\leq r_i:=\dim V_i-1$. Define rational maps
    $$
    \phi_i:U\dashrightarrow\mathbb P^{r_i},\quad z\longmapsto(T_{j_1,\ldots,j_{i-1},0,j_{i+1},\ldots,j_\ell}(z):\cdots:T_{j_1,\ldots,j_{i-1},r_i,j_{i+1},\ldots,j_\ell}(z))
    $$
    for $j_1,\ldots,j_{i-1},j_{i+1},\ldots,j_\ell$ varying. They are well defined, for $T$ is $X$-simple (see \Cref{Rmk:XSimple}). Since $U$ is smooth, there are line bundles $L_1,\ldots,L_\ell$ on $U$ such that $\phi_i=(s_{i,0}:\cdots:s_{i,r_i})$ for some movable linear systems $|s_{i,0},\ldots,s_{i,r_i}|\subseteq |L_i|$. Now consider the local sections 
    $$
    \frac{T_{j_1,\ldots,j_\ell}}{s_{1,j_1}\cdots s_{\ell,j_\ell}},\quad 0\leq j_i\leq r_i,
    $$ 
    of $L\otimes L_1^{-1}\otimes\cdots\otimes L_\ell^{-1}$. They coincide in the function field of $U$, and the set of the denominators $s_{i,j_1}\cdots s_{\ell,j_\ell}$ cuts out a subscheme of $U$ with no divisorial parts since so do the $|s_{i,0},\ldots,s_{i,r_i}|$, $1\leq i\leq\ell$. Hence, the smoothness of $U$ yields a nonzero global section $t\in H^0(L\otimes L_1^{-1}\otimes\cdots\otimes L_\ell^{-1})$ fitting into
    $$
    T_{j_1,\ldots,j_\ell}=s_{1,j_1}\cdots s_{\ell,j_\ell}t
    $$
    for all $0\leq j_i\leq r_i$. By identifying $V_i=\langle s_{i,0},\ldots,s_{i,r_i}\rangle$ this expression shows all of the assertions.
\end{proof}

\subsection{Symmetric powers of smooth curves}

Throughout this subsection we assume $C$ to be a smooth projective curve of genus $g$ and discuss line bundles and divisors on the $(q+1)$-th \emph{symmetric power} $C_{q+1}$ of $C$, that is, the quotient 
$$
C_{q+1}=C^{q+1}/\mathfrak S_{q+1}
$$
of the $(q+1)$-th Cartesian power $C^{q+1}$ of $C$ by the factor-permuting action of $\mathfrak S_{q+1}$ on $C^{q+1}$.
It is in fact the Hilbert scheme of effective divisors $\xi$ of degree $q+1$ on $C$, being a smooth projective variety of dimension $q+1$. 

We list notions and facts about $C_{q+1}$. Let $B\in\Pic(C)$ stand for a line bundle on $C$.

\begin{enumerate}
\item Write $B^{\boxtimes q+1}=B\boxtimes\cdots\boxtimes B$ for the $(q+1)$-th box power of $B$ on $C^{q+1}$.
\item We denote by $S_{q+1,B}$ the line bundle on $C_{q+1}$ to which $B^{\boxtimes q+1}$ descends with respect to the canonical $\mathfrak S_{q+1}$-linearization. For a point $z\in C$ the effective divisor $X_{q+1}^z:=\{\xi\in C_{q+1}:\xi\ni z\}$ lies in $|S_{q+1,\OO_C(z)}|$.
\item Let $a_{q+1}:C\times C_q\to C_{q+1}$ be the addition map $(z,\xi)\mapsto z+\xi$. It acts as the universal family of degree $q+1$ effective divisors on $C$.
\item The pushforward $E_{q+1,B}:=a_{q+1,\ast}(B\boxtimes\OO_{C_q})$ is a vector bundle of rank $q+1$ on $C_{q+1}$ whose fiber over each $\xi\in C_{q+1}$ is $H^0(B|_\xi)$.
\item Put $N_{q+1,B}=\det E_{q+1,B}$. Its global sections are $H^0(N_{q+1,B})=\A^{q+1}H^0(B)$, and furthermore, $H^i(N_{q+1,B})=\A^{q+1-i}H^0(B)\otimes S^iH^1(B)$ for all $0\leq i\leq q+1$. 
\item Nonreduced effective divisors of degree $q+1$ on $C$ form an effective divisor on $C_{q+1}$, namely the diagonal divisor $\Delta_{q+1}\subset C_{q+1}$.
\item There is a divisor $\delta_{q+1}$ on $C_{q+1}$ such that $2\delta_{q+1}$ is linearly equivalent to $\Delta_{q+1}$. Then $N_{q+1,B}\cong S_{q+1,B}(-\delta_{q+1})$.
\item We have the natural isomorphism $\Delta_2\cong C$ under which $\OO_{C_2}(-\delta_2)|_{\Delta_2}\cong\omega_C$.
\item The product $C_{q+1}\times C^i$ has an effective divisor 
$$
D_{q+1}^i:=\{(\xi,z_1,\ldots,z_i):\xi\ni z_j\textup{ for some }1\leq j\leq i\}.
$$
Its fibers are given as
\begin{align*}
(D_{q+1}^i)_\xi & =\xi\times C^{i-1}+\cdots+C^{i-1}\times\xi\in|\OO_C(\xi)^{\boxtimes i}|\quad\textup{and} \\
(D_{q+1}^i)_{(z_1,\ldots,z_i)} & =X_{q+1}^{z_1}+\cdots+X_{q+1}^{z_i}\in|S_{q+1,\OO_C(z_1+\cdots+z_i)}|.
\end{align*}
Also, $C^i$ has a similar effective divisor 
$$
D^{i-1,1}:=\{(z_1,\ldots,z_{i-1},z_i)\in C^i:z_j=z_i\textup{ for some }1\leq j\leq i-1\}.
$$
\end{enumerate}

For these matters see \cite{MR151460}, \cite{MR770932}, \cite{MR1158344}, \cite{MR1149124}, and  \cite{MR4160876}.

\begin{Prop}
For an integer $q\geq 0$ if a line bundle $B$ on $C$ is $q$-very ample, then $E_{q+1,B}$, hence $N_{q+1,B}$, is globally generated.
\end{Prop}

\begin{proof}
Consider the evaluation map
$$
\operatorname{ev}:H^0(E_{q+1,B})\otimes\OO_{C_{q+1}}\to E_{q+1,B}.
$$
Notice that $H^0(E_{q+1,B})=H^0(B)$ and that the fiber of $\operatorname{ev}$ over each $\xi\in C_{q+1}$ is
$$
\operatorname{ev}_\xi:H^0(B)\to H^0(B|_\xi).
$$
Hence, from the $q$-very ampleness of $B$ if follows that $E_{q+1,B}$ is globally generated.
\end{proof}

Let $B\in\Pic(C)$ be a line bundle on $C$ and $q\geq0$ be an integer. Suppose that $B$ is $q$-very ample so that $E_{q+1,B}$ is globally generated. Let $M_{q+1,B}$ be the kernel bundle in the short exact sequence
$$
\begin{tikzcd}
0 \ar[r] & M_{q+1,B} \ar[r] & H^0(B)\otimes\OO_{C_{q+1}} \ar[r] & E_{q+1,B} \ar[r] & 0.
\end{tikzcd}
$$
Taking $\A^{q+1}$ we have a locally free resolution
\begin{equation}\label{Eqn:NResolution}
\begin{tikzcd}
F_{q+1}(B)_i:=S^iM_{q+1,B}\otimes\A^{q+1-i}H^0(B) \ar[r] & \cdots \ar[r] & F_{q+1}(B)_1 \ar[r] & F_{q+1}(B)_0
\end{tikzcd}    
\end{equation}
of $\det E_{q+1,B}=N_{q+1,B}$.

For the proof of \Cref{Thm:EKS} we need to compute some sheaf cohomology groups related to the resolutions $F_{q+1}(B)_\ast$ and an $(i-1)$-very ample line bundle $L_1$ on $C$ for an integer $i\geq 1$, which employs the pullback
$$
Q^i_{L_1}:=\pi^\ast_iM_{i,L_1}
$$ 
via the quotient map $\pi_i:C^i\to C_i$. Notice that $Q^1_{L_1}$ is the well-known kernel bundle $M_{L_1}:=\ker(H^0(L_1)\otimes\OO_C\to L_1)$ and that the short exact sequence
\begin{equation}\label{Eqn:RathmannSES}
\begin{tikzcd}
    0 \ar[r] & Q^i_{L_1} \ar[r] & Q^{i-1}_{L_1}\boxtimes\OO_C \ar[r] & \OO_{C^{i-1}}\boxtimes L_1(-D^{i-1,1}) \ar[r] & 0
\end{tikzcd}    
\end{equation}
\cite[p.\ 3]{rathmann2016effective} holds. For convenience in the following a line bundle $B$ on $C$ is called \emph{Clifford} if
$$
B\in
\begin{cases}
    \emptyset & \textup{if $g=0$} \\
    \{\OO_C,\omega_C\} & \textup{if $C$ is not hyperelliptic with $g\geq 1$} \\
    \{G^m:0\leq m\leq g-1\} & \textup{if $C$ is hyperelliptic with $|G|=g^1_2$}.
\end{cases}
$$

\begin{Prop}\label{Prop:VanishingLB}
Let $i\geq1$ and $q\geq 1$ be integers, and consider line bundles $L_1$ and $B_1$ on $C$ such that $L_1$ is $(i-1)$-very ample, $\deg B_1\geq 2g+q$, and
\begin{equation}\label{Cond:NotClifford}
\deg L_1+\deg B_1\geq 4g+2q\quad\text{with equality only if}\quad B_1\otimes L_1^{-1}\text{ is not Clifford.}    
\end{equation}
Then for all $B_2,\ldots,B_i\in\Pic(C)$ we obtain the vanishing
$$
H^i(C^i,\A^{q+j+1}Q^i_{L_1}\otimes B_1(\xi)\boxtimes B_2(\xi)\boxtimes\cdots\boxtimes B_i(\xi))=0
$$
provided that an effective divisor $\xi\in C_j$ of degree $j\geq0$ satisfies 
\begin{equation}\label{Cond:ImposeIndependent}
h^0(C,L_1\otimes B_1^{-1}\otimes\omega_C(-\xi))=\max\{h^0(C,L_1\otimes B_1^{-1}\otimes\omega_C)-j,0\}.    
\end{equation}
\end{Prop}

\begin{proof}
We use induction on $i\geq1$. For the initial stage $i=1$ we proceed as in the proofs of \cite[Corollary (4.e.4)]{MR739785} and \cite[Theorem 2]{MR944326}. To the contrary we assume that $H^1(\A^{q+j+1}Q^1_{L_1}\otimes B_1(\xi))\neq0$. 
Then by Serre duality $H^0(\A^{h^0(L_1)-q-j-2}Q^1_{L_1}\otimes L_1\otimes B_1^{-1}\otimes\omega_C(-\xi))=H^1(\A^{q+j+1}Q^1_{L_1}\otimes B_1(\xi))^\ast$ is nonzero, and so the well-known vanishing theorem \cite[Theorem (3.a.1)]{MR739785} due to Green gives
$$
0\leq h^0(L_1)-q-j-2\leq h^0(L_1\otimes B_1^{-1}\otimes\omega_C(-\xi))-1.
$$
If $h^0(L_1\otimes B_1^{-1}\otimes\omega_C)\leq j$, then Condition \eqref{Cond:ImposeIndependent} deduces $h^0(L_1\otimes B_1^{-1}\otimes\omega_C(-\xi))=0$, a contradiction. Thus, we have $h^0(L_1\otimes B_1^{-1}\otimes\omega_C(-\xi))=h^0(L_1\otimes B_1^{-1}\otimes\omega_C)-j\geq 1$ and so
$$
h^0(L_1)-q-1\leq h^0(L_1\otimes B_1^{-1}\otimes\omega_C)
$$
together with $H^0(L_1\otimes B_1^{-1}\otimes\omega_C)\neq 0$. If $H^1(L_1\otimes B_1^{-1}\otimes\omega_C)=0$, then
\begin{align*}
    h^0(L_1\otimes B_1^{-1}\otimes\omega_C) & =(\deg L_1-\deg B_1+2g-2)-g+1 \\
    & =(\deg L_1-g+1)-\deg B_1+2g-2 \\
    & \leq h^0(L_1)-\deg B_1+2g-2.
\end{align*}
In this case we reach $\deg B_1\leq 2g+q-1$, a contradiction. As a result, Clifford's theorem on special divisors says that
\begin{align*}
\deg L_1-\deg B_1+2g-2 & \geq 2(h^0(L_1\otimes B_1^{-1}\otimes\omega_C)-1) \\
& \geq 2(h^0(L_1)-q-2) \\
& \geq 2((\deg L_1-g+1)-q-2) \\
& =2\deg L-2g-2q-2.
\end{align*}
So we obtain 
$$
\deg L_1+\deg B_1\leq 4g+2q,
$$
a contraction by Condition \eqref{Cond:NotClifford} and the sharp case of Clifford's theorem. We have finished the case $i=1$.

Now we suppose the induction hypothesis assigned to an integer $i-1\geq 1$. Applying $\A^{q+j+1}$ to \eqref{Eqn:RathmannSES} we have
$$
\begin{tikzcd}
    0 \ar[r] & \A^{q+j+2}Q^i_{L_1} \ar[r] & \A^{q+j+2}Q^{i-1}_{L_1}\boxtimes\OO_C \ar[r] & \A^{q+j+1}Q^i_{L_1}\otimes(\OO_{C^{i-1}}\boxtimes L_1)(-D^{i-1,1}) \ar[r] & 0.
\end{tikzcd}
$$
By tensoring it with $B_1(\xi)\boxtimes\cdots\boxtimes B_{i-1}(\xi)\boxtimes(B_i(\xi)\otimes L_1^{-1})(D^{i-1,1})$, it is enough to show that
$$
H^i\left(\left(\A^{q+j+2}Q^{i-1}_{L_1}\otimes B_1(\xi)\boxtimes\cdots\boxtimes B_{i-1}(\xi)\right)\boxtimes\left(B_i(\xi)\otimes L_1^{-1}\right)(D^{i-1,1})\right)=0.
$$
In turn, the Leray spectral sequence for the projection map $\pr_2:C^{i-1}\times C\to C$ reduces it to proving that $H^1(R^{i-1}\otimes B_i(\xi)\otimes L_1^{-1})=0$, where
$$
R^{i-1}:=R^{i-1}\pr_{2,\ast}\left(\A^{q+j+2}Q^{i-1}_{L_1}\otimes B_1(\xi)\boxtimes\cdots\boxtimes B_{i-1}(\xi)\right)\boxtimes\OO_C(D^{i-1,1}).
$$
Note that the last vanishing holds when $R^{i-1}$ has finite support and that the base change map of $R^{i-1}$ has the form
$$
R^{i-1}\otimes\mathbb C(z)\to H^{i-1}(\A^{q+j+2}Q^{i-1}_{L_1}\otimes B_1(\xi+z)\boxtimes\cdots\boxtimes B_{i-1}(\xi+z))
$$
over any point $z\in C$. Hence, by cohomology and base change it suffices to obtain
$$
H^{i-1}(\A^{q+j+2}Q^{i-1}_{L_1}\otimes B_1(\xi+z)\boxtimes\cdots\boxtimes B_{i-1}(\xi+z))=0
$$
for a general point $z\in C$. The induction process works.
\end{proof}

We close this subsection with a computation of higher direct images, which enables us to argue, via the Leray spectral sequence, a kind of sheaf cohomology vanishing around the incidence divisor $D_{q+1}^i\subset C_{q+1}\times C^i$.

\begin{Lem}[cf.\ {\cite[Lemma 3.3]{choe2023syzygies}}]\label{Lem:Leray}
Let $L_1,B_1,\ldots,B_i\in\Pic(C)$ be line bundles on $C$, and suppose that $\deg L_1\geq 2g+i-1$ and $\deg B_k\geq 2g+q$ for all $1\leq k\leq i$. Consider the following projection maps.
$$
\begin{tikzcd}
    & C_{q+1}\times C^i \ar[dl,"\pr_1",swap] \ar[dr,"\pr_2"] \\
    C_{q+1} & & C^i
\end{tikzcd}
$$
Then one has
\begin{align*}
R^j\pr_{1,\ast}\OO_{C_{q+1}}\boxtimes(B_1\boxtimes\cdots\boxtimes B_i)(-D_{q+1}^i) & =
\begin{cases}
    M_{q+1,B_1}\otimes\cdots\otimes M_{q+1,B_i} & \textup{if }j=0 \\
    0 & \textup{otherwise}
\end{cases}
\quad\text{and}\\
R^j\pr_{2,\ast}N_{q+1,L_1}\boxtimes\OO_{C^i}(-D_{q+1}^i) & =
\begin{cases}
    \A^{q+1}Q^i_{L_1} & \textup{if }j=0 \\
    0 & \textup{otherwise}.
\end{cases}
\end{align*}
\end{Lem}

\begin{proof}
We provide a proof only for the first equality as the second one follows in a similar manner. For the argument below observe that
$$
h^j(B_k(-\xi))=
\begin{cases}
    h^0(B_k)-q-1 & \textup{if }j=0 \\
    0 & \textup{otherwise}
\end{cases}
$$
for all $j,k\in\mathbb N$ and $\xi\in C_{q+1}$ because of the degree bounds $\deg B_k\geq 2g+q$. Put $R^j=R^j\pr_{1,\ast}\OO_{C_{q+1}}\boxtimes(B_1\boxtimes\cdots\boxtimes B_i)(-D_{q+1}^i)$ for simplicity. If $j=0$, then the inclusion
\begin{align*}
\pr_{1,\ast}\OO_{C_{q+1}}\boxtimes(B_1\boxtimes\cdots\boxtimes B_i)(-D_{q+1}^i) & \subseteq\pr_{1,\ast}\OO_{C_{q+1}}\boxtimes(B_1\boxtimes\cdots\boxtimes B_i) \\
&=H^0(B_1)\otimes\cdots\otimes H^0(B_i)\otimes\OO_{C_{q+1}}
\end{align*}
is given, its base change maps look like
$$
H^0(B_1(-\xi))\otimes\cdots\otimes H^0(B_i(-\xi))\subseteq H^0(B_1)\otimes\cdots\otimes H^0(B_i),
$$
and the subspaces have the same dimension. Hence, in this case $R^0$ turns out to be a subbundle of $H^0(B_1)\otimes\cdots\otimes H^0(B_i)\otimes\OO_{C_{q+1}}$ by cohomology and base change, and the only possibility is
$$
R^0=M_{q+1,B_1}\otimes\cdots\otimes M_{q+1,B_i}.
$$
If $j\geq 1$, then the base change map of $R^j$ is computed as
\begin{align*}
    R^j\otimes\mathbb C(\xi) & \to H^j(B_1(-\xi)\boxtimes\cdots\boxtimes B_i(-\xi)) \\
    & =\bigoplus_{j_1+\cdots+j_i=j}H^{j_1}(B_1(-\xi))\otimes\cdots\otimes H^{j_i}(B_i(-\xi)) \\
    & =0
\end{align*}
via the K{\"u}nneth formula, and consequently, we have $R^j=0$ again by cohomology and base change.
\end{proof}

\section{Green-Lazarsfeld classes revisited}

In this section we reformulate the \emph{Green-Lazarsfeld classes} using multilinear algebra. This approach sheds light on their useful substantial structures.

We begin with a preparatory step. In what follows for simplicity we use the Einstein summation convention: ``repeated indices are implicitly summed over." For example, we may write
\begin{equation}\label{Eqn:fG}
f=\frac{1}{\alpha!}u^\alpha\otimes f_\alpha\quad\text{and}\quad G=\frac{1}{\beta!}w^\beta\otimes G_\beta,
\end{equation}
where $|\alpha|=a$ and $|\beta|=b+1$, for arbitrary elements $f\in S^aU\otimes\A^{b+q+1}U$ and $G\in S^{b+1}W\otimes\A^{a+q}W$ with respect to given bases $u$ and $w$ of $U$ and $W$, respectively. Also, we ignore the scaling of nonzero coefficients when it causes no issues.

\begin{Def}
We define a product
$$
\boxtimes^0=\boxtimes_{(a,b),q+1}^0:(S^aU\otimes\A^{b+q+1}U)\times(S^{b+1}W\otimes\A^{a+q}W)\to\A^{a+b+1}(U\otimes W)\otimes S^q(U\otimes W)
$$
as follows. Let $(f,G)$ be an arbitrary pair in the domain. Consider the maps
$$
\begin{tikzcd}
S^aU\otimes\A^{b+q+1}U \ar[r,"\id\otimes\Delta"] & S^aU\otimes\A^{b+1}U\otimes\A^qU
\end{tikzcd}
$$
and
$$
\begin{tikzcd}
S^{b+1}W\otimes\A^{a+q}W \ar[r,"\id\otimes\Delta"] & S^{b+1}W\otimes\A^aW\otimes\A^qW \ar[r,equal] & \A^aW\otimes S^{b+1}W\otimes\A^qW.
\end{tikzcd}
$$
We identify $f$ and $G$ with their corresponding images so that
$$
f\otimes G\in\A^a(U\otimes W)\otimes\A^{b+1}(U\otimes W)\otimes S^q(U\otimes W)
$$
up to the map $\edet\otimes\edet\otimes\det$. Going through the wedge product $\wedge$ on the factor $\A^a(U\otimes W)\otimes\A^{b+1}(U\otimes W)$, we obtain
$$
f\boxtimes^0 G\in\A^{a+b+1}(U\otimes W)\otimes S^q(U\otimes W).
$$

Similarly, under the natural interchange one may establish another product
$$
\boxtimes_0=\boxtimes^{(a,b),q+1}_0:(\A^aU\otimes S^{b+q+1}U)\times(S^{b+1}W\otimes\A^{a+q}W)\to S^{a+b+1}(U\otimes W)\otimes\A^q(U\otimes W).
$$

\end{Def}

The products $\boxtimes^0$ and $\boxtimes_0$ have the following key properties in relation to some subspaces of $K_{a+b,q+1}(U,W)$ and $K^{a+b,q+1}(U,W)$. 

\begin{Def}
We introduce
\begin{align*}
K_{p,q+1}(U,W) & =Z_{p,q+1}(U\otimes W)\cap\left(\A^p(U\otimes W)\otimes(\A^{q+1}U\otimes\A^{q+1}W)\right)\quad\text{and} \\
K^{p,q+1}(U,W) & =Z^{p,q+1}(U\otimes W)\cap\left(S^p(U\otimes W)\otimes(S^{q+1}U\otimes\A^{q+1}W)\right).
\end{align*}
\end{Def}

\begin{Rmk}
The space $K_{p,q+1}(U,W)$ is equal to the bottom syzygy space $K_{p,q+1}(I(\sigma_qY))$ for the Segre variety $Y:=\mathbb PU\times\mathbb PW\subseteq\mathbb P(U\otimes W)$.
\end{Rmk}

\begin{Prop}\label{Prop:Product}
Let $f$ and $G$ be as in \eqref{Eqn:fG}. Then 
$$
\overline{f\boxtimes^0G}:=\delta_{a+b+1,q}(f\boxtimes^0G)\in Z_{a+b,q+1}(U\otimes W)
$$ 
satisfies the following.
\begin{enumerate}
    \item\label{Prop:ProductItem1} $\overline{f\boxtimes^0G}\in K_{a+b,q+1}(U,W)$.
    \item\label{Prop:ProductItem2} If we fix $f\in Z^{a,b+q+1}(U)$, then $\overline{f\boxtimes^0G}$ depends only on $\delta^{b+1,a+q}(G)\in Z^{b,a+q+1}(W)$.
    \item\label{Prop:ProductItem3} If $G=w_0^{b+1}\otimes w_{\{1,\ldots,a+q\}}$ provided that $\dim W\geq a+q+1$, then $\overline{f\boxtimes^0G}$ involves a nonzero scalar multiple of
    $$
\partial^u_{\overline{\alpha}'}f_\alpha\otimes w_{\{0,a+1,\ldots,a+q\}}\in\A^{q+1}U\otimes\A^{q+1}W
$$
as the cofactor of
    $$
    (u^\alpha\otimes w_{\{1,\ldots,a\}})\wedge\left(u_{\overline{\alpha}'}\otimes w_0^b\right)\in\A^{a+b}(U\otimes W),
    $$
    for all subsets $\alpha,\overline{\alpha}'\subseteq\{0,\ldots,\dim U-1\}$ with $|\alpha|=a$ and $|\overline{\alpha}'|=b$.
\end{enumerate}
Their counterparts for $\boxtimes_0=\boxtimes^{(a,b),q+1}_0$ are also valid.
\end{Prop}

\begin{proof}
\ref{Prop:ProductItem1} One has
$$
    \overline{f\boxtimes^0G} = \underbrace{\left((\partial^u_i\otimes\partial^w_j)\left((u^\alpha\otimes w_{\overline{\beta}})\wedge(u_{\overline{\alpha}}\otimes w^\beta)\right)\right)}_{=:(A)}\otimes\underbrace{\left(\frac{1}{\alpha!\beta!}(u_i\otimes w_j)\cdot(\partial^u_{\overline{\alpha}}f_\alpha\otimes\partial^w_{\overline{\beta}}G_\beta)\right)}_{=:(B)},
$$
where $|\overline{\alpha}|=b+1$ and $|\overline{\beta}|=a$, and the Leibniz rule gives
$$
(A)=\underbrace{(\partial_u^iu^\alpha\otimes\partial^w_jw_{\overline{\beta}})\wedge(u_{\overline{\alpha}}\otimes w^\beta)}_{=:(A_1)}\pm\underbrace{(u^\alpha\otimes w_{\overline{\beta}})\wedge(\partial^u_iu_{\overline{\alpha}}\otimes\partial_w^jw^\beta)}_{=:(A_2)}.
$$
Then the term $(A_1)\otimes(B)$ lies in $\A^{a+b}(U\otimes W)\otimes(\A^{q+1}U\otimes\A^{q+1}W)$ since
\begin{flushleft}
\quad$\displaystyle{\frac{1}{\alpha!}(\partial_u^iu^\alpha\otimes\partial^w_jw_{\overline{\beta}})\otimes
\left((u_i\otimes w_j)\cdot(\partial^u_{\overline{\alpha}}f_\alpha\otimes\partial^w_{\overline{\beta}}G_\beta)\right)}$
\end{flushleft}\vspace*{-\baselineskip}
\begin{align*}
    & =\frac{1}{\alpha'!}(u^{\alpha'}\otimes w_{\overline{\beta}'})\otimes
    \left((u_i\otimes
    w_j)\cdot(\partial^u_{\overline{\alpha}}f_{\alpha'+e_i}\otimes(\partial^w_j\wedge\partial^w_{\overline{\beta}'})G_\beta)\right) \tag{$\alpha':=\alpha-e_i$ and $\overline{\beta}':=\overline{\beta}\setminus\{j\}$} \\
    & = \pm\frac{1}{\alpha'!}(u^{\alpha'}\otimes w_{\overline{\beta}'})\otimes
    \left((u_i\otimes w_j)\cdot(\partial^u_{\overline{\alpha}}f_{\alpha'+e_i}\otimes\partial^w_j\partial^w_{\overline{\beta}'}G_\beta)\right) \\
    & = \pm\frac{1}{\alpha'!}(u^{\alpha'}\otimes w_{\overline{\beta}'})\otimes \left((u_i\wedge\partial^u_{\overline{\alpha}}f_{\alpha'+e_i})\otimes(w_j\wedge\partial^w_j\partial^w_{\overline{\beta}'}G_\beta)\right) \tag{\Cref{Lem:FactorWedge} with $\partial^w_{\overline{\beta}'}G_\beta$ on $\A^{q+1}W$} \\
    & = \pm\frac{1}{\alpha'!}(u^{\alpha'}\otimes w_{\overline{\beta}'})\otimes
    \left((u_i\wedge\partial^u_{\overline{\alpha}}f_{\alpha'+e_i})\otimes \partial^w_{\overline{\beta}'}G_\beta\right), \tag{\Cref{Prop:Euler}}
\end{align*}
where $e_0,e_1,\ldots$ are the standard unit vectors. Similarly, so does $(A_2)\otimes(B)$ due to
\begin{flushleft}
\quad$\displaystyle{\frac{1}{\beta!}(\partial^u_iu_{\overline{\alpha}}\otimes\partial_w^jw^\beta)
\otimes\left((u_i\otimes w_j)\cdot(\partial^u_{\overline{\alpha}}f_\alpha\otimes\partial^w_{\overline{\beta}}G_\beta)\right)}$
\end{flushleft}\vspace*{-\baselineskip}
\begin{align*}
    & =\pm\frac{1}{\beta'!}(u_{\overline{\alpha}'}\otimes w^{\beta'})\otimes  \left(\partial^u_{\overline{\alpha}'}f_\alpha\otimes(w_j\wedge\partial^w_{\overline{\beta}}G_{\beta'+e_j})\right). \tag{$\overline{\alpha}':=\overline{\alpha}\setminus\{i\}$ and $\beta':=\beta-e_j$}
\end{align*}

\ref{Prop:ProductItem2} This follows from \Cref{Thm:Lascoux} together with
$$
Z^{a,b+q+1}(U)=S^{(a+1,1,\ldots,1)}U\quad\text{and}\quad\overline{f\boxtimes^0G}\in K_{a+b,q+1}(I(\sigma_qY))
$$
for the partition $(a+1,1,\ldots,1)\vdash a+b+q+1$ and the Segre variety $Y:=\mathbb PU\times\mathbb PW\subseteq\mathbb P(U\otimes W)$. For $\boxtimes_0$ see \Cref{Lem:LascouxDual} below. 

\ref{Prop:ProductItem3} We have
$$
\partial^w_{\overline{\beta}}G_\beta=
\begin{cases}
    \pm w_{\{1,\ldots,a+q\}\setminus\overline{\beta}}& \textup{if }\beta=(b+1)e_0\textup{ and }\overline{\beta}\subseteq\{1,\ldots,a+q\} \\
    0 & \textup{otherwise}.
\end{cases}
$$ 
Observe that the factors of $(A_1)\otimes(B)$ on $\A^{a+b}(U\otimes W)$ consist of $(u^{\alpha'}\otimes w_{\overline{\beta}'})\wedge(u_{\overline{\alpha}}\otimes w_0^{b+1})$ with $\overline{\beta}'\subseteq\{1,\ldots,a+q\}$, which can not have the form $(u^\alpha\otimes w_{\{1,\ldots,a\}})\wedge(u_{\overline{\alpha}'}\otimes w_0^b)$ and that the factors of $(A_2)\otimes(B)$ on $\A^{a+b}(U\otimes W)$ are $(u^\alpha\otimes w_{\overline{\beta}})\wedge(u_{\overline{\alpha}'}\otimes w_0^b)$ with $\overline{\beta}\subseteq\{1,\ldots,a+q\}$. Thus, the desired term is
$$
\left((u^\alpha\otimes w_{\{1,\ldots,a\}})\wedge(u_{\overline{\alpha}'}\otimes w_0^b)\right)\otimes
(\partial^u_{\overline{\alpha}'}f_\alpha\otimes w_{\{0,a+1,\ldots,a+q\}}).
$$
We are done.
\end{proof}

\begin{Lem}\label{Lem:LascouxDual}
For every integer $c\geq0$ we obtain a containment of the form
$$
K^{c,q+1}(U,W)\subseteq\bigoplus_{b=0}^cS^{(b+q+1,1,\ldots,1)}U\otimes S^{(b+1,1,\ldots,1)}W,
$$
where all the partitions are of $c+q+1$.
\end{Lem}

\begin{proof}
Taking an auxiliary vector space $\overline{U}$ of dimension $\geq c+q+1$ as in the proof of \Cref{Lem:WedgeProlong}, one can see that
$$
\A^{c+q+1}\overline{U}\otimes K^{c,q+1}(U,W)\subseteq K_{c,q+1}(\overline{U}\otimes U,W).
$$
Use the decomposition of $K_{c,q+1}(\overline{U}\otimes U,W)$ (\Cref{Thm:Lascoux} in the case $m=1$) with the help of \Cref{Thm:CL} to complete the proof.
\end{proof}

These features give rise to new products derived from the previous ones.

\begin{Def}
Let $U$ and $W$ be vector spaces. By
\begin{align*}
\boxtimes=\boxtimes_{(a,b),q+1}: & Z^{a,b+q+1}(U)\times Z^{b,a+q+1}(W)\to K_{a+b,q+1}(U,W)\quad\text{and} \\
\boxtimes=\boxtimes^{(a,b),q+1}: & Z_{a,b+q+1}(U)\times Z^{b,a+q+1}(W)\to K^{a+b,q+1}(U,W),
\end{align*}
we denote the products that are guaranteed by \Cref{Prop:Product}\ref{Prop:ProductItem1} and \ref{Prop:ProductItem2}.
\end{Def}

Note that $K_{p,q+1}(U,W)$ is nothing but $K_{p,q+1}(I(\sigma_qY))$ for the Segre variety $Y:=\mathbb PU\times\mathbb PW\subseteq\mathbb P(U\otimes W)$.

We now conclude this section with an interpretation of the Green–Lazarsfeld classes, which also provides a concrete description of the part $m=1$ of the Lascoux resolutions (\Cref{Thm:Lascoux}).

\begin{Thm}[Green-Lazarsfeld classes]\label{Thm:GreenLazarsfeld}
Let $c\geq 0$ and $q\geq 1$ be integers. Then the products $\boxtimes$ give decompositions
\begin{align*}
K_{c,q+1}(U,W) & =\bigoplus_{a+b=c}Z^{a,b+q+1}(U)\otimes Z^{b,a+q+1}(W)\quad\text{and} \\
K^{c,q+1}(U,W) & =\bigoplus_{a+b=c}Z_{a,b+q+1}(U)\otimes Z^{b,a+q+1}(W).
\end{align*}
Also, the following statements and their interchange partners hold.
\begin{enumerate}
    \item\label{Thm:GreenLazarsfeldItem1} Let $f\in Z^{a,b+q+1}(U)$ be an element and $B\subseteq\A^{q+1}U$ be a subspace. If $f\boxtimes g\in\A^{a+b}(U\otimes W)\otimes(B\otimes\A^{q+1}W)$ for all $g\in Z^{b,a+q+1}(W)$, then $f$ consists of the $b$-th prolongation $B^{(b)}$ on the factor $\A^{b+q+1}U$, that is, $f\in S^aU\otimes B^{(b)}$.
    \item\label{Thm:GreenLazarsfeldItem2} Let $M:U\otimes W\to V'$ be a matrix that can be realized as a multiplication map of line bundle global sections on a nonempty open subset of $X$. Then the induced map $K_{a+b,q+1}(U,W)\to Z_{a+b,q+1}(V')$ sends any product of the form
    $$
    f\boxtimes g:=f\boxtimes(w_0^b\otimes w_{\{0,1,\ldots,a+q\}}),\quad 0\neq f\in Z^{a,b+q+1}(U),
    $$
    to a nonzero cycle in $Z_{a+b,q+1}(V')$, where $\{w_0,w_1,\ldots\}$ is a suitable basis of $W$.
\end{enumerate}
\end{Thm}

\begin{proof}
We verify the first decomposition. To this end by \Cref{Thm:Lascoux} with $m=1$ it is enough to show that the product $\boxtimes$ forms an injection 
$$
\boxtimes:Z^{a,b+q+1}(U)\otimes Z^{b,a+q+1}(W)\hookrightarrow K_{a+b,q+1}(U,W)
$$
provided that $Z^{a,b+q+1}(U)$ and $Z^{b,a+q+1}(W)$ are both nonzero. Then \Cref{Prop:Product}\ref{Prop:ProductItem3} implies that the map has a nonzero image. Thus, we are done in the manner of \Cref{Prop:TrivialitySchurtoSchur}. As for the second decomposition, similarly, $K^{c,q+1}(U,W)$ contains all the summands in question up to the mappings of $\boxtimes$, and by \Cref{Lem:LascouxDual} no other components exist.

\ref{Thm:GreenLazarsfeldItem1} For this part we apply \Cref{Prop:Product}\ref{Prop:ProductItem3} to find that 
$$
\partial^u_{\overline{\alpha}'}f_\alpha\in B
$$ 
for all $\alpha\in\mathbb N^{\dim U}$ and $\overline{\alpha}'\subseteq\{0,\ldots,\dim U-1\}$ with $|\alpha|=a$ and $|\overline{\alpha}'|=b$, hence $f_\alpha\in B^{(b)}$ for every $\alpha\in\mathbb N^{
\dim U}$ by the definition of $B^{(b)}$. 

\ref{Thm:GreenLazarsfeldItem2} Let $z_1,\ldots,z_a\in X$ be general points, and choose bases $u$ and $w$ of $U$ and $W$, respectively, in such a way that $u_i(z_j)\neq0$ (resp.\ $w_i(z_j)\neq0$) if and only if $i=j$. Now consider a basis $x$ of $V'$ satisfying that 
\begin{center}
$x_i=M(u_i\otimes w_i)$ for all $1\leq i\leq a$, and the others vanish at $z_1,\ldots,z_a$.      
\end{center}
Then differentiating $f\boxtimes g$ with $\partial^a/\partial x_1\cdots\partial x_a$ on the factor $\A^{q+1}V'$, we have
\begin{equation}\label{Eqn:GeometricNonvanishing}
\frac{\partial^a}{\partial x_1\cdots\partial x_a}(f\boxtimes g)=(u_{\overline{\alpha}'}\otimes w_0^b)\otimes(\partial^u_{\overline{\alpha}'}f_{e_1+\cdots+e_a}\otimes w_{0,a+1,\ldots,a+q})
\end{equation}
up to scaling. Keeping in mind \Cref{Thm:Eisenbud1Generic} we find that the vanishing of \eqref{Eqn:GeometricNonvanishing} is equivalent to that of $f_{e_1+\cdots+e_a}$. However, since $z_1,\ldots,z_a\in X$ have been general, the form $f_{e_1+\cdots+e_a}$ is nonzero, which implies that so is $f\boxtimes g$.
\end{proof}

\section{Proofs}

In this section we prove the main results.

\subsection{Proof of \Cref{Thm:SegreDecompo}}
More generally, we show that
$$
K_{p,q+1}(V_1,\ldots,V_\ell)=\bigoplus_{p_\ast\vdash p}B_{p_\ast,q+1}(V_1,\ldots,V_\ell)
$$
for every $\ell\geq 1$, where
$$
K_{p,q+1}(V_1,\ldots,V_\ell):=\begin{cases}  Z_{p,q+1}(V_1\otimes\cdots\otimes V_\ell)\cap\left(\A^p(V_1\otimes\cdots\otimes V_\ell)\otimes(\A^{q+1}V_1\otimes\cdots\otimes\A^{q+1}V_\ell)\right) & \textup{if }2\mid\ell \\
Z^{p,q+1}(V_1\otimes\cdots\otimes V_\ell)\cap\left(S^p(V_1\otimes\cdots\otimes V_\ell)\otimes(\A^{q+1}V_1\otimes\cdots\otimes\A^{q+1}V_\ell)\right) & \textup{if }2\nmid\ell.
\end{cases}
$$
This suffices because of \Cref{Cor:SegreBottomEquation} which also shows \Cref{ThmItem:SegreDecompoTrivial}. 

We proceed by induction on $\ell\geq 1$. For the initial step $\ell=1$ there is nothing to prove. Set $\ell\geq 2$, and for simplicity assume that $\ell$ is even; the odd case is parallel to the even one. Since $\A^{q+1}V_1\otimes\cdots\otimes\A^{q+1}V_{\ell-1}\otimes\A^{q+1}V_\ell\subseteq\A^{q+1}(V_1\otimes\cdots\otimes V_{\ell-1})\otimes\A^{q+1}V_\ell$, we have
$$
K_{p,q+1}(V_1,\ldots,V_{\ell-1},V_\ell)\subseteq K_{p,q+1}(V_1\otimes\cdots\otimes V_{\ell-1},V_\ell)= \bigoplus_{a+b=p}Z^{a,b+q+1}(V_1\otimes\cdots\otimes V_{\ell-1})\otimes Z^{b,a+q+1}(V_\ell)
$$
by \Cref{Thm:GreenLazarsfeld} up to the products $\boxtimes$. Let $M\subseteq K_{p,q+1}(V_1,\ldots,V_{\ell-1},V_\ell)$ be an irreducible submodule over $\GL(V_1)\times\cdots\times\GL(V_{\ell-1})\times\GL(V_\ell)$. Obviously, it has the form 
$$
M=M'\otimes Z^{b,a+q+1}(V_\ell)
$$
with $M'\subseteq Z^{a,b+q+1}(V_1\otimes\cdots\otimes V_{\ell-1})$ irreducible over $\GL(V_1)\times\cdots\times\GL(V_{\ell-1})$. Remind that $M'\otimes Z^{b,a+q+1}(V_\ell)\subseteq\A^p(V_1\otimes\cdots\otimes V_{\ell-1}\otimes V_\ell)\otimes(\A^{q+1}V_1\otimes\cdots\otimes\A^{q+1}V_{\ell-1}\otimes\A^{q+1}V_\ell)$. \Cref{Thm:GreenLazarsfeld}\ref{Thm:GreenLazarsfeldItem1} and \Cref{Lem:WedgeProlong} imply that 
\begin{align*}
M' & \subseteq Z^{a,b+q+1}(V_1\otimes\cdots\otimes V_{\ell-1})\cap\left(S^a(V_1\otimes\cdots\otimes V_{\ell-1})\otimes(\A^{b+q+1}V_1\otimes\cdots\otimes\A^{b+q+1}V_{\ell-1})\right) \\
& = K_{a,b+q+1}(V_1,\ldots,V_{\ell-1}).
\end{align*}
That is, we reach
$$
K_{p,q+1}(V_1,\ldots,V_{\ell-1},V_\ell)\subseteq\bigoplus_{a+b=p}K_{a,b+q+1}(V_1,\ldots,V_{\ell-1})\otimes Z^{b,a+q+1}(V_\ell).
$$
On the other hand, by \Cref{Thm:GreenLazarsfeld}\ref{Thm:GreenLazarsfeldItem2} each summand indeed appears in $K_{p,q+1}(V_1,\ldots,V_{\ell-1},V_\ell)$. In conclusion, we have a decomposition
$$
K_{p,q+1}(V_1,\ldots,V_{\ell-1},V_\ell)=\bigoplus_{a+b=p}K_{a,b+q+1}(V_1,\ldots,V_{\ell-1})\otimes Z^{b,a+q+1}(V_\ell),
$$
making the induction work.

To verify \Cref{Thm:SegreDecompo}\ref{ThmItem:SegreDecompoUpper} we assume that $B_{p_\ast,q+1}(V_1,\ldots,V_{2k+2})\neq 0$ for some ordered partition $p_\ast=(p_1,\ldots,p_{2k+2})\vdash p$, that is, the inequalities in Condition \eqref{Cond:Nonvanishing} all hold. Sum them up to obtain 
$$
\sum_{i=1}^{2k+2}\dim V_i\geq (2k+1)p+(2k+2)(q+1).
$$
In other words, if $(2k+1)p>\sum_i(\dim V_i-q-1)$, then $B_{p_\ast,q+1}(V_1,\ldots,V_{2k+2})=0$ always.

\subsection{Proofs of \Cref{Thm:Main1,,Thm:Main2}}

We provide a single proof that covers both \Cref{Thm:Main1,,Thm:Main2}. We begin with the following commutative diagram.
$$
\begin{tikzcd}
    X \ar[r,dashed,"T^\ast|_X"] \ar[d,phantom,"{\rotatebox[origin=c]{-90}{$\subseteq$}}"] & \mathbb PV_1\times\cdots\times\mathbb PV_{2k+2} \ar[d,phantom,"{\rotatebox[origin=c]{-90}{$\subseteq$}}"] \\
    \mathbb PV \ar[r,dashed,"T^\ast",swap] & \mathbb P(V_1\otimes\cdots\otimes V_{2k+2})
\end{tikzcd}
$$
Note that the bottom map $T^\ast:\mathbb PV\dashrightarrow\mathbb P(V_1\otimes\cdots\otimes V_{2k+2})$ is linear. Thus, pulling back bottom syzygies of $\sigma_q\tau^kY$ for $Y:=\mathbb PV_1\times\cdots\times\mathbb PV_{2k+2}\subseteq\mathbb P(V_1\otimes\cdots\otimes V_{2k+2})$, we obtain those of $\sigma_q\tau^kX\subseteq\mathbb PV$:
$$
b_{p_\ast,q+1}(T):B_{p_\ast,q+1}(V_1,\ldots,V_{2k+2})\to K_{p,q+1}(I(\sigma_q\tau^kX)),
$$
and such bottom syzygies can be nontrivial as ensured by \Cref{Thm:GreenLazarsfeld}\ref{Thm:GreenLazarsfeldItem2}. To be more specific, if for each $1\leq i\leq 2k+2$ we take an appropriate basis $u$ of $V_i$ and put
$$
f_i=u_0^{p_i}\otimes u_{\{0,\ldots,p-p_i+q\}}\in Z^{p_i,p-p_i+q+1}(V_i),
$$
then $f_1\otimes\cdots\otimes f_{2k+2}$ maps to a nonzero bottom syzygy in $K_{p,q+1}(I(\sigma_q\tau^kX))$ through the products $\boxtimes$.

\begin{Rmk}
There is a different approach to \Cref{Thm:Main1}. Keep in mind \Cref{Prop:StructureXmulti}, \Cref{Thm:BottomEqn}, and their notations, and observe that
$$
\A^{q+1}V_i\subseteq H^0(L_i^{\boxtimes q+1}\otimes\II_{\Delta^{q+1}/U^{q+1}})
$$
for all $1\leq i\leq 2k+2$. So we are given the multiplication map
$$
\A^{q+1}V_1\otimes\cdots\otimes\A^{q+1}V_{2k+2}\to H^0(L^{\boxtimes q+1}\otimes\II_{\Delta^{q+1}/U^{q+1}}^{2k+2})
$$
which factors through $S^{q+1}V$. We are done by \Cref{Thm:BottomEqn}.
\end{Rmk}

\subsection{Proof of \Cref{Thm:EKS}}

In the spirit of \Cref{Thm:BottomEqn} we find that
$$
I(\sigma_q\tau^kC)_{q+1}=H^0(S_{q+1,L}(-(2k+2)\delta_{q+1})).
$$
(One could examine the short exact sequences
$$
\begin{tikzcd}
    0 \ar[r] & S_{2,L}(-(i+1)\Delta_2) \ar[r] & S_{2,L}(-i\Delta_2) \ar[r] & L^2\otimes\omega_{\Delta_2}^{2i} \ar[r] & 0
\end{tikzcd}
$$
for the sake of certainty.) So we show that for every integer $\ell\geq 2$ if $\deg L\geq\ell(2g+q)$, then a multiplication map
$$
H^0(N_{q+1,L_1})\otimes\cdots\otimes H^0(N_{q+1,L_\ell})\to H^0(S_{q+1,L}(-\ell\delta_{q+1}))
$$
is surjective for some factorization $L_1\otimes\cdots\otimes L_\ell=L$ in $\Pic(C)$.

Suppose that $\deg L\geq\ell(2k+q)$. We may decompose $L$ as 
$$
L=L_1\otimes\cdots\otimes L_\ell
$$
in $\Pic(C)$ so that
\begin{enumerate}
    \item $\deg L_i\geq 2g+q$ for all $1\leq i\leq\ell$, and
    \item if $g\geq 1$, and if $\deg L_1=\deg L_2=2g+q$, then $L_1$ is not isomorphic to $L_2$.
\end{enumerate}
Indeed, let $L_3,\ldots,L_\ell$ be any line bundles of degree $\geq 2g+q$ such that $L_{1,2}:=L\otimes L_3^{-1}\otimes\cdots\otimes L_\ell^{-1}$ has degree $\geq 4g+2q$. Now if $\deg L_{1,2}>4g+2q$, then choose an arbitrary line bundle $L_2$ of degree $2g+q\leq\deg L_2\leq\deg L_{1,2}-(2g+q)$, but if $\deg L_{1,2}=4g+2q$, then pick a general line bundle $L_2$ of degree $2g+q$ so that $L_2^2\not\cong L_{1,2}$ in the case $g\geq1$. Our procedure ends with $L_1:=L_{1,2}\otimes L_2^{-1}$.

For such a choice of $L_1,\ldots,L_\ell$ let $F_{q+1}(\widehat{L_1})_\ast$ be the total complex of $F_{q+1}(L_2)_\ast\otimes\cdots\otimes F_{q+1}(L_\ell)_\ast$ (see \eqref{Eqn:NResolution}), that is,
$$
F_{q+1}(\widehat{L_1})_i=\bigoplus_{i_2+\cdots+i_\ell=i}S^{i_2}M_{q+1,L_2}\otimes\cdots\otimes S^{i_\ell}M_{q+1,L_\ell}
$$
with constant bundle factors omitted, and $F_{q+1}(\widehat{L_1})_0=H^0(N_{q+1,L_2})\otimes\cdots\otimes H^0(N_{q+1,L_\ell})$. Observe that $N_{q+1,L_1}\otimes F_{q+1}(\widehat{L_1})_\ast$ is a resolution of $N_{q+1,L_1}\otimes N_{q+1,L_2}\otimes\cdots\otimes N_{q+1,L_\ell}=S_{q+1,L}(-\ell\delta_{q+1})$ and that $H^0(N_{q+1,L_1}\otimes F_{q+1}(\widehat{L_1})_0)=H^0(N_{q+1,L_1})\otimes H^0(N_{q+1,L_2})\otimes\cdots\otimes H^0(N_{q+1,L_\ell})$. 
Therefore, it is enough to obtain the vanishing 
$$
H^i(N_{q+1,L_1}\otimes F_{q+1}(\widehat{L_1})_i)=0
$$ 
for every $i\geq 1$.

Since $S^{i_j}M_{q+1,L_j}$ is a direct summand of $T^{i_j}M_{q+1,L_j}$ for each $2\leq j\leq \ell$, the desired vanishing in turn follows from the vanishing
\begin{equation}\label{Eqn:VanishingLB}
H^i(N_{q+1,L_1}\otimes M_{q+1,B_1}\otimes\cdots\otimes M_{q+1,B_i})=0    
\end{equation}
for line bundles $B_1,\ldots,B_i$ of degree $\geq 2g+q$ such that if $g\geq1$, and if $\deg L_1=\deg B_1=2g+q$, then $L_1\not\cong B_1$. Using the Leray spectral sequences for the projection maps 
$$
\pr_1:C_{q+1}\times C^i\to C_{q+1}\quad\textup{and}\quad\pr_2:C_{q+1}\times C^i\to C^i
$$ 
in conjunction with \Cref{Lem:Leray}, we find that \eqref{Eqn:VanishingLB} is equivalent to the vanishing
$$
H^i(\A^{q+1}Q^i_{L_1}\otimes B_1\boxtimes\cdots\boxtimes B_i)=0.
$$
Applying \Cref{Prop:VanishingLB} we are done. 

\subsection{Proof of \Cref{Cor:Byproduct}}

By \Cref{Thm:Main1} the tangent variety $\tau X$ has a quadratic equation, but $\sigma_2X$ can not by its nature. Thus, $\tau X$ is strictly included in $\sigma_2X$, which means that 
$$
\dim\tau X=2\dim X\quad\text{and}\quad\dim\sigma_2X=2\dim X+1
$$
due to \cite[Corollary 4]{MR541334}.

\section{Comments and questions}

This section collect comments and questions on this study.

\subsection{Tangent varieties to smooth curves}

Let $C$ be a smooth projective curve of genus $g$ and $C\subseteq\mathbb PH^0(L)$ be a complete embedding by a very ample line bundle $L$ on $C$.

This subsection focuses on the quantity
$$
\ell(\tau C):=\max\{p\geq 0:K_{p,2}(I(\tau C))\neq 0\}.
$$

The rational case $g=0$ has been thoroughly studied and understood, which indeed forms the core of the result by Aprodu–Farkas–Papadima–Raicu–Weyman mentioned in the introduction.

\begin{Thm}[{\cite[Theorem 5.4]{MR4022070}}] Let $C=\mathbb PU\subseteq\mathbb PS^dU$ be the rational normal curve of degree $d\geq 3$, where $\dim U=2$. Then for every integer $p\geq 0$ the Koszul cohomology group $K_{p,2}(I(\tau C))$ is isomorphic to the kernel of the following composition map
$$
\begin{tikzcd}
    S^{2p+2}U\otimes S^{d-p-3}(S^{p+2}U) \ar[rr] \ar[rd,"\gamma_p^\ast\otimes\textup{id}",swap] & & S^{p+2}U\otimes S^{d-p-2}(S^{p+2}U) \\
    & \A^{2}(S^{p+2}U)\otimes S^{d-p-3}(S^{p+2}U) \ar[ru,"\delta_{2,d-p-3}",swap]
\end{tikzcd}
$$
up to $\det U$ factors, where $\gamma_p:\A^{2}(S^{p+2}U^\ast)\to\otimes S^{2p+2}U^\ast$ is the Gaussian map $y^a\wedge y^b\mapsto (a-b)y^{a+b-1}$ for an affine coordinate $y$ on $\mathbb PU^\ast$: $U^\ast=\langle 1,y\rangle$ and $S^{p+2}U^\ast=\langle 1,y,\ldots,y^{p+2}\rangle$. 
\end{Thm}

As a corollary, a sophisticated computation of what are known as the Koszul modules yields
$$
\ell(\tau C)=\lfloor d/2\rfloor-2
$$
when $C$ is rational.

With regard to nonrational curves $C$, \Cref{Thm:Main2} gives rise to the lower bound
\begin{equation}\label{Eqn:1/3LB}
\ell(\tau C)\geq\left\lfloor\frac{d-\gon(C)}{3}\right\rfloor-g-1   
\end{equation}
when $L$ has degree $d\geq 3g+\gon(C)+3$, where $\gon(C)$ is the gonality of $C$. Indeed, let $p$ be the right-hand side value in \eqref{Eqn:1/3LB}, and choose line bundles $L_1,\ldots,L_4$ on $C$ so that 
\begin{enumerate}
    \item $L_1$ is a line bundle computing the gonality, and
    \item $L_2,L_3,L_4$ have degree $\geq g+p+1$ under the constraint that $L_2\otimes L_3\otimes L_4=L\otimes L_1^{-1}$.
\end{enumerate}
Then since
$$
h^0(L_i)\geq
\begin{cases}
    2 & \textup{if }i=1 \\
    p+2 & \textup{if }2\leq i\leq 4,
\end{cases}
$$
Condition \eqref{Cond:Nonvanishing} holds for the ordered partition $p_\ast=(p,0,0,0)\vdash p$. In addition, it is conjectured in \cite[Conjecture 4.6]{MR4896737} that
$$
\ell(\tau C)\leq\left\lfloor\frac{d-3g}{2}\right\rfloor-2,
$$
which is based on some computational evidence.

This discussion brings us to the following question.

\begin{Q}
What is the correct asymptotics of $\ell(\tau C)$? How about the vanishing/nonvanishing of $K_{p,q+1}(I(\sigma_q\tau^kC))$?
\end{Q}
One can find that if $\deg L\geq (2k+1)(g+p+q)+\gon^q(C)$ for an integer $p\geq0$, where $\gon^q(C):=\min\{d\geq 1:C\textup{ carries a }g_d^q\}$, then $K_{p,q+1}(I(\sigma_q\tau^kC))\neq0$.

On the other hand, let us look at the initial stage of the minimal free resolution of $I(\tau C)$. Due to \Cref{Thm:EKS} we have already found the source of all the degree $q+1$ equations of $\sigma_q\tau^kC$, but in order to complete the full matryoshka picture of Eisenbud–Koh–Stillman's work, we need to understand not merely the single component $I(\sigma_q\tau^kC)_{q+1}$, but all components of the ideal.

\begin{Q}
Suppose that $L$ is sufficiently positive. Is $I(\sigma_q\tau^kC)$ generated in degree $q+1$?
\end{Q}

Park made a guess \cite[Conjecture 4.7 with (4-3)]{MR4896737} that it is the case for $\tau C\subseteq\mathbb PH^0(L)$ if
$$
 \deg L\geq 5g+5,
$$
and more generally, if $\deg L\geq 4g+3+(g+2)p$ for an integer $p\geq 1$, then $\tau C\subseteq\mathbb PH^0(L)$ satisfies property $N_p$, which means that $K_{i,j}(I(\tau C))=0$ whenever $i\leq p-1$ and $j\geq 3$.

\subsection{Syzygies over exterior algebras}

The theory of syzygies (over the symmetric algebra $S^\ast V$) naturally generalizes to the exterior algebra $E:=\A^{\ast}V$. Let $P$ be a finitely generated graded $E$-module, and consider the $i$-th free module 
$$
\bigoplus_{j\in\mathbb Z}K^{i,j}(P)\otimes E(-i-j)
$$
of the minimal graded $E$-free resolution of $M$, where $K^{i,j}(P)$ is the cohomology group of the Koszul type complex
$$
\begin{tikzcd}
    S^{i+1}V\otimes P_{j-1} \ar[r] & S^iV\otimes P_j \ar[r] & S^{i-1}V\otimes P_{j+1}
\end{tikzcd}
$$
at the middle. Let $F\subseteq\Gr(\mathbb PV):=\bigsqcup_{s=0}^{\dim\mathbb PV}\Gr(\mathbb P^s,\mathbb PV)$ be a subvariety of a Grassmannian for $\mathbb PV$. To it we assign a homogeneous ideal
$$
J(F)=\{w\in E:w|_\Lambda=0\text{ for all }[\Lambda]\in F\}\subseteq E,
$$
where we denote by $w|_\Lambda$ the image of $w$ under the restriction map $\A^{\ast}H^0(\OO_{\mathbb PV}(1))\to\A^{\ast}H^0(\OO_\Lambda(1))$.

Returning to our central objects we define some subvarieties of $\Gr(\mathbb PV)$ by abuse of notation.

\begin{Def}
We write
$$
\sigma_q\tau^kX=\overline{\{[\langle\mathbb T^k_{z_1}X,\ldots,\mathbb T^k_{z_q}X\rangle]:z_1,\ldots,z_q\in X\textup{ are general}\}}\subset\Gr(\mathbb PV).
$$
\end{Def}

Then with the induced exterior ideals $J(\sigma_q\tau^kX)$ our theorems generalize in the exterior setting. For example, we state the following.

\begin{Thm}[Analogue of {\Cref{Thm:Main2}}]
Let $p\geq0$, $q\geq 1$, and $k\geq 0$ be integers and $p_\ast=(p_1,\ldots,p_{2k+1})\vdash p$ be an ordered partition of $p$.  If $T:V_1\otimes\cdots\otimes V_{2k+1}\to V$ is an $X$-multiplicative $(2k+1)$-way tensor of linear forms, and if
$$
\dim V_1\geq p-p_1+q+1,\quad\ldots,\quad\dim V_{2k+1}\geq p-p_{2k+1}+q+1,
$$
then a nontrivial canonical map
$$
b_{p_\ast,q+1}(T):B_{p_\ast,q+1}(V_1,\ldots,V_{2k+1})\to K^{p,q+1}(J(\sigma_q\tau^kX))
$$
is given.
\end{Thm}

Also, $K^{p,q+1}(U,W)$ has a meaning in terms of exterior syzygies. Let $Y=\mathbb PU\times\mathbb PW\subseteq\mathbb P(U\otimes W)$ be the Segre variety with two factors, and define
$$
\sigma_q(Y/\mathbb PW)=\{\mathbb P(U\otimes\overline{W}):\overline{W}\twoheadleftarrow W\textup{ is a quotient of dimension }q\}\subset\Gr(\mathbb P(U\otimes W)).
$$

\begin{Prop}
Let $Y$ be as above, and take integers $p\geq 0$ and $q\geq 1$. Then $K^{p,q+1}(J(\sigma_q(Y/\mathbb PW)))$ is isomorphic to $K^{p,q+1}(U,W)$.    
\end{Prop}

Accordingly, it can be said that our proofs implicitly pass through the exterior world. In particular, the Green-Lazarsfeld products $\boxtimes$ and their symmetric-symmetric versions lead to the following decompositions that display hidden interplays between symmetric and exterior bottom syzygies of $\sigma_q\tau^kY$ for $(2k+2)$- or $(2k+1)$-factor Segre varieties $Y$ with $q\geq1$ and $k\geq0$ varying. 
\begin{enumerate}
\item If $Y=Y_1\times Y_2:=(\mathbb PU_1\times\cdots\times\mathbb PU_{2k_1+1})\times(\mathbb PW_1\times\cdots\times\mathbb PW_{2k_2+1})\subseteq\mathbb P(U_1\otimes\cdots\otimes U_{2k_1+1}\otimes W_1\otimes\cdots\otimes W_{2k_2+1})$ has $2k+2:=2(k_1+k_2)+2$ factors of projective spaces, then
$$
K_{c,q+1}(I(\sigma_q\tau^kY))=\bigoplus_{a+b=c}K^{a,b+q+1}(J(\sigma_{b+q+1}\tau^{k_1}Y_1))\otimes K^{b,a+q+1}(J(\sigma_{a+q+1}\tau^{k_2}Y_2)).
$$
\item If $Y=Y_1\times Y_2:=(\mathbb PU_1\times\cdots\times\mathbb PU_{2k_1+2})\times(\mathbb PW_1\times\cdots\times\mathbb PW_{2k_2+1})\subseteq\mathbb P(U_1\otimes\cdots\otimes U_{2k_1+2}\otimes W_1\otimes\cdots\otimes W_{2k_2+1})$ has $2k+1:=2(k_1+k_2+1)+1$ factors of projective spaces, then
$$
K^{c,q+1}(J(\sigma_q\tau^kY))=\bigoplus_{a+b=c}K_{a,b+q+1}(I(\sigma_{b+q+1}\tau^{k_1}Y_1))\otimes K^{b,a+q+1}(J(\sigma_{a+q+1}\tau^{k_2}Y_2)).
$$
\item If $Y=Y_1\times Y_2:=(\mathbb PU_1\times\cdots\times\mathbb PU_{2k_1+2})\times(\mathbb PW_{3}\times\cdots\times\mathbb PW_{2k_2+2})\subseteq\mathbb P(U_1\otimes\cdots\otimes U_{2k_1+2}\otimes W_1\otimes\cdots\otimes W_{2k_2+2})$ has $2k+2:=2(k_1+k_2+1)+2$ factors of projective spaces, then
$$
K_{c,q+1}(I(\sigma_q\tau^kY))=\bigoplus_{a+b=c}K_{a,b+q+1}(I(\sigma_{b+q+1}\tau^{k_1}Y_1))\otimes K_{b,a+q+1}(I(\sigma_{a+q+1}\tau^{k_2}Y_2)).
$$
\end{enumerate}

\bigskip

\noindent{\textbf{Acknowledments}.} We would like to thank Daniele Agostini and Jinhyung Park for valuable discussions and for kindly sharing their manuscript. The author is supported by a KIAS Individual Grant (MG083302) at Korea Institute for Advanced Study.

\bibliographystyle{amsalpha}
\bibliography{references}

\end{document}